\documentclass{amsart}%
\usepackage{amsfonts}
\usepackage{amssymb}
\usepackage{amsmath}%
\usepackage{graphicx}
\newtheorem{theorem}{Theorem}[section]
\theoremstyle{plain}

\newtheorem{corollary}[theorem]{Corollary}

\newtheorem{lemma}{Lemma}[section]

\newtheorem{remark}{Remark}[section]

\numberwithin{equation}{section}
\begin{document}
\title[$p$-ADIC PROPERTIES OF TRANSLATED DIVISION\ POLYNOMIALS]{$p$-ADIC PROPERTIES OF TRANSLATED DIVISION\ POLYNOMIALS and SOMOS\ SEQUENCES}
\author{BET\"{U}L GEZER, OSMAN\ B\.{I}Z\.{I}M}
\address{Bursa Uludag University, Faculty of Science, Department of Mathematics,
G\"{o}r\"{u}kle, 16059, Bursa-TURKEY}
\email{betulgezer@uludag.edu.tr, obizim@uludag.edu.tr}
\date{27. 10. 2026.}
\maketitle

\begin{abstract}
We study the periodicity and $p$-adic convergence of sequences of values of
the translated division polynomials $(\Psi_{n}(\mathbf{P}))_{n\geq0}$,
$(\Phi_{n}(\mathbf{P}))_{n\geq0}$, and $(\overline{\Omega}\,_{n}%
(\mathbf{P}))_{n\geq0}$ evaluated at a pair of points on an elliptic curve. We
prove that these sequences are purely periodic over finite fields. We then use
their periodicity properties to prove that certain subsequences are $%
\mathbb{Z}
_{p}$-Cauchy. Applying this result to Somos $4$ and Somos $5$ sequences yields
analogous convergence results. We also establish pure periodicity modulo prime
powers for the even and odd indexed subsequences of Somos 5 sequences under
suitable hypotheses.

\end{abstract}
\tableofcontents

\section{Introduction}

Let $E/K$ denote an elliptic curve defined over a field $K$ given by a
Weierstrass model $f(x,y)=0$ where
\begin{equation}
f(x,y)=y^{2}+a_{1}xy+a_{3}y-x^{3}-a_{2}x^{2}-a_{4}x-a_{6}\text{.}\label{1}%
\end{equation}
For background on elliptic curves, see \cite{JS1} and \cite{JS2}. Let $E(K)$
denote the group of $K$-rational points on the elliptic curve $E$ together
with the point at infinity $\mathcal{O}$ as the identity element. Assume that
$\mathbf{P}=(P_{1},P_{2})\in$ $E(K)^{2}$. Write
\begin{equation}
\lbrack n]P_{1}+P_{2}=\left(  \frac{\Phi_{n}(\mathbf{P})}{\Psi_{n}%
(\mathbf{P})^{2}}\text{, }\frac{\overline{\Omega}\,_{n}(\mathbf{P})}{\Psi
_{n}(\mathbf{P})^{3}}\right) \label{2}%
\end{equation}
for "polynomials"$\ \Phi_{n}$, $\Psi_{n}$, $\overline{\Omega}\,_{n}\in
K[x_{i},y_{i}][(x_{1}-x_{2})^{-1}]/\left\langle f(x_{i},y_{i})\right\rangle $,
$i=1$, $2$, whenever $[n]P_{1}+P_{2}\neq\mathcal{O}$. The translated division
polynomials describe the points $[n]P_{1}+P_{2}$, playing a role analogous to
that of the usual division polynomials $\psi_{n}(P)$ in describing the
multiples $[n]P_{1}$ (see \cite[Exercise 3.7]{JS1} and \cite[Chapter 2]{SL}
for basic properties of division polynomials). Division polynomials are
closely related to elliptic divisibility sequences (EDSs), whose arithmetic
properties were studied by Ward \cite{MW1, MW2}. We refer the reader to
\cite{GE1}, \cite{JS} for definitions and further background on elliptic
sequences and EDSs.

Stange \cite{SKT,SK} generalized the notion of elliptic sequences to
multidimensional arrays, called \textit{elliptic nets}. These encode linear
combinations of points on an elliptic curve through rational functions called
\textit{net polynomials}. The translated division polynomials considered here
correspond to the indices $(n,1)$ in rank two; their construction and
evaluation are recalled in Section 2. She \cite{SKT} proved that the elliptic
nets demonstrate certain partial periodicity properties, which were named
\textit{symmetry properties} after Ward. See also \cite{AK}, for symmetries of
elliptic nets and valuations of elliptic net polynomials. For related symmetry
results for sequences, see Ayad \cite{MA} and Swart \cite{CS}. Silverman
\cite{JS3} used a lift to characteristic zero and the Lefschetz principle to
prove that the sequence $(\Psi_{n}(P))_{n\geq0}$ over finite fields is a
purely periodic sequence. He also considered the convergence properties of the
sequence $(\Psi_{n}(P))_{n\geq0}$ when $K$ is a complete local field and
showed that certain subsequences of the sequence converge in $%
\mathbb{Z}
_{p}$. In \cite{GB}, the authors followed Silverman's approach in an addendum
to \cite{JS3} to prove that the sequences $(\Phi_{n}(P))_{n\geq0}$ and
$(\overline{\Omega}_{n}(P))_{n\geq0}$ are periodic over finite fields and that
certain subsequences are $%
\mathbb{Z}
_{p}$-Cauchy.

Motivated by these results, we study the periodicity and $p$-adic convergence
properties of the sequences of values of the translated division polynomials
$(\Psi_{n}(\mathbf{P}))_{n\geq0}$, $(\Phi_{n}(\mathbf{P}))_{n\geq0}$, and
$(\overline{\Omega}\,_{n}(\mathbf{P}))_{n\geq0}$. The periodicity of
$(\Psi_{n}(\mathbf{P}))_{n\geq0}$ over finite fields was previously
established by Swart \cite{CS}. Here, we give an alternative argument inspired
by Silverman's method \cite[Theorem 8]{JS3} and derive corresponding
periodicity results for $(\Phi_{n}(\mathbf{P}))_{n\geq0}$ and $(\overline
{\Omega}\,_{n}(\mathbf{P}))_{n\geq0}$. We further give $p$-adic convergence
results for suitable subsequences of these translated division polynomial
sequences and apply them to Somos $4$ and Somos $5$ sequences. In particular,
we obtain pure periodicity results for the even and odd indexed subsequences
of Somos $5$ sequences under appropriate $p$-adic conditions.

The following theorem states that the sequences of values of the translated
division polynomials are purely periodic over finite fields.

\begin{theorem}
Let ${\mathbb{F}}$ be a finite field, let $E/{\mathbb{F}}$ be an elliptic
curve, let $\mathbf{P}=(P_{1}$, $P_{2})\in$ $E({\mathbb{F}})^{2}$ be a pair of
points such that the order of $P_{1}$ is $r>3$ and $P_{2}$, $P_{1}\pm
P_{2}\neq$ $\mathcal{O}$. Then, the sequences $(\Psi_{n}(\mathbf{P}))_{n\geq
0}$, $(\Phi_{n}(\mathbf{P}))_{n\geq0}$, and $(\overline{\Omega}\,_{n}%
(\mathbf{P}))_{n\geq0}$ are purely periodic with periods $rt$, $rt^{\prime}$,
and $rt^{\prime\prime}$ respectively, for some integers $t$, $t^{\prime}$, and
$t^{\prime\prime}$ each dividing $q-1$, where $q$ is the order of the field
${\mathbb{F}}$.
\end{theorem}

The proof for the sequence $(\Psi_{n}(\mathbf{P}))_{n\geq0}$ is given in
Section 3, while the proofs for the sequences $(\Phi_{n}(\mathbf{P}))_{n\geq
0}$ and $(\overline{\Omega}\,_{n}(\mathbf{P}))_{n\geq0}$ are given in Section
4. We first give a symmetry formula for $(\Psi_{n}(\mathbf{P}))_{n\geq0}$, and
then derive the corresponding formulas for $(\Phi_{n}(\mathbf{P}))_{n\geq0}$
and $(\overline{\Omega}\,_{n}(\mathbf{P}))_{n\geq0} $ from the coordinate
identities in Section 2.

In the statements involving reduction modulo powers of $p$, we assume that the
coefficients of the Weierstrass equation (\ref{1}) belong to $\mathbb{Z}_{p}$.
All reductions are taken with respect to this model.

The following theorem shows that the sequences $(\Phi_{n}(\mathbf{P}%
))_{n\geq0}$ and $(\overline{\Omega}\,_{n}(\mathbf{P}))_{n\geq0}$ are also
periodic modulo prime powers.

\begin{theorem}
Let $E/%
\mathbb{Q}
$ be an elliptic curve and let $\mathbf{P}=(P_{1},P_{2})\in$ $E(\mathbb{Q}%
)^{2}$ be a pair consisting of rational points of infinite order. Let $p$ be
an odd prime at which the chosen integral Weierstrass model of $E$ has good
reduction, and let $r>3$ be the order of $P_{1}$ modulo $p$. Assume further
that $P_{2}$, $P_{1}\pm P_{2}$ $\not \equiv \mathcal{O}$ $(\operatorname{mod}$
$p)$. Let $p^{v}$ be the highest power of $p$ such that $[r]P_{1}%
\equiv\mathcal{O} $ $(\operatorname{mod}$ $p^{v})$, and write $r_{l}=p^{l-v}r$
for $l\geq v$ and $r_{l}=r$ for $l<v$. Then, for all $l\geq1$, the sequences
\[
(\Phi_{_{n}}(\mathbf{P})\operatorname{mod}p^{l})_{n\geq1}\text{ \ and
\ }(\overline{\Omega}\,_{_{n}}(\mathbf{P})\operatorname{mod}\text{ }%
p^{l})_{n\geq1}%
\]
are purely periodic with periods $r_{l}t_{l}$ and $r_{l}t_{l}^{\prime}$
respectively, for some positive integers $t_{l}$ and $t_{l}^{\prime}$ dividing
$p^{l-1}(p-1)$.
\end{theorem}

The integer $r_{l}$ in Theorem 1.2 is the order of $P_{1}$ modulo $p^{l}$; see
\cite[Theorem 2.1]{MW1} and \cite[Theorem 3.9.5]{CS}.

A result of Swart \cite[Theorem 9.4.1]{CS} and Theorem 1.2 lead to an
interesting question: whether these sequences possess a $p$-adic convergence
property. This question was studied by Silverman \cite[Theorem 2]{JS3} for the
sequences of values of the division polynomials $(\Psi_{n}(P))_{n\geq0} $, and
in an addendum he connected it with Ayad's work \cite{MA}. In \cite[Theorem
1.3]{GB}, the authors considered the same question for the sequences of values
of division polynomials $(\Phi_{n}(P))_{n\geq0}$ and $(\overline{\Omega}%
_{n}(P))_{n\geq0}$. The following theorem provides generalizations of
\cite[Theorem 2]{JS3} and \cite[Theorem 1.3]{GB} for the sequences $(\Psi
_{n}(\mathbf{P}))_{n\geq1}$, $(\Phi_{n}(\mathbf{P}))_{n\geq1}$ and
$(\overline{\Omega}\,_{n}(\mathbf{P}))_{n\geq1}$ of values of the translated
division polynomials of $E$ at $\mathbf{P}\in$ $E(K)^{2}$. In particular, the
theorem shows that certain $p$-power-indexed subsequences of these sequences
converge in $\mathbb{Z}_{p}$.

\begin{theorem}
Let $p$ be an odd prime and let $E/%
\mathbb{Q}
_{p}$ be an elliptic curve given by a Weierstrass equation over $\mathbb{Z}%
_{p}$ with good reduction. Let $\mathbf{P}=(P_{1},P_{2})\in$ $E(%
\mathbb{Q}
_{p})^{2}$ be a pair of points such that the reduction of $P_{1}$ modulo $p$
has order $r>3$, with $p\nmid r$. Assume further that $P_{2}$, $P_{1}\pm
P_{2}$ $\not \equiv \mathcal{O}$ $(\operatorname{mod}$ $p)$. Then there exists
a power $q=p^{N}$ such that for all $m\geq1$, the limits
\[
\lim\limits_{i\rightarrow\infty}\Psi_{mq^{i}}(\mathbf{P})\text{, \ }%
\lim\limits_{i\rightarrow\infty}\Phi_{mq^{i}}(\mathbf{P})\text{ and }%
\lim\limits_{i\rightarrow\infty}\overline{\Omega}\,_{mq^{i}}(\mathbf{P})
\]
exist in $\mathbb{Z}_{p}$.
\end{theorem}

As a consequence of Theorem 1.3, we obtain several results regarding the
$p$-adic convergence of Somos sequences, which were introduced by M. Somos.
For $k\geq4$, a\textit{\ Somos }$k$\textit{\ sequence} is a sequence
$(W(n))_{n\geq0}~$that satisfies the recurrence relation%
\begin{equation}
W(n)W(n-k)=\sum\limits_{i=1}^{\left\lfloor k/2\right\rfloor }\tau
_{i}W(n-i)W(n-k+i)\label{1.1}%
\end{equation}
where the $\tau_{i}$ are constant parameters. See \cite{GE1}, \cite{GI},
\cite{Robinson}, \cite{AH1}, \cite{Propp}, \cite{CS}, \cite{P}, for more
details on Somos sequences.

The connection between Somos $4$ sequences and elliptic curves was discovered
independently by Elkies and Stephens; see \cite{Propp}, \cite{CS}. Hone's
sigma-function formula \cite[Theorem 1.1]{AH1} gives%
\[
W(n)=W(0)\left(  \frac{W(1)}{W(0)}\right)  ^{n}\Psi_{n}(\mathbf{P})
\]
where $\mathbf{P}=(P_{1},P_{2})$ is the associated pair of points on $E$. This
representation connects the $p$-adic behaviour of a Somos $4$ sequence with
that of the translated division polynomial values and a geometric factor.
Combining it with Theorem 1.3 yields the following result.

\begin{theorem}
Let $(W(n))_{n\geq0}$ be a Somos $4$ sequence with $\tau_{1}$, $W(0)\neq0$.
Let $E/%
\mathbb{Q}
$ be the associated elliptic curve and let $\mathbf{P}=(P_{1},P_{2})\in$
$E(\mathbb{Q})^{2}$ be the associated pair of rational points. Let $p$ be an
odd prime at which the chosen integral Weierstrass model of $E$ has good
reduction. Let $P_{1}$ be a point whose reduction modulo $p$ has order $r>3$,
assume that $p\nmid r$, and let $P_{1}$, $P_{2}$, $P_{2}\pm P_{1}\not \equiv $
$\mathcal{O}$ $(\operatorname{mod}p)$. Assume further that ord$_{p}\left(
W(0)\right)  \geq0$, ord$_{p}\left(  \frac{W(1)}{W(0)}\right)  $ $\geq0$. Then
there is a power $q=p^{N}$ such that for all $m\geq1$, the limit
\[
\lim\limits_{i\rightarrow\infty}W(mq^{i})\text{ \ }%
\]
exists in $\mathbb{Z}_{p}$.
\end{theorem}

\begin{remark}
The reduction hypotheses ensure that the translated division polynomial values
are $p$-adic integers; see the integrality discussion at the end of Section 2.
The additional valuation hypotheses in the Somos applications ensure the
$p$-adic integrality of the resulting sequences, as explained in Section 7.
\end{remark}

Hone \cite[Theorem 2.9]{AH3} obtained analogous representations for the even
and odd indexed subsequences of a Somos 5 sequence. Each subsequence is a
geometric factor times a sequence of translated division polynomial values.
Applying Theorem 1.3 to the two associated pairs of points gives the following result.

\begin{theorem}
Let $(W(n))_{n\geq0}$ be a Somos $5$ sequence with $W(0)W(1)\neq0$ and let $E/%
\mathbb{Q}
$ be the associated elliptic curve. Let $\mathbf{P}=(P_{1},P_{2})$ and
$\mathbf{Q}=(Q_{1},Q_{2})\in$ $E(\mathbb{Q})^{2}$ be pairs of rational points
associated to the even and odd indexed subsequences of $(W(n))$, respectively.
Let $p$ be an odd prime at which the chosen integral Weierstrass model of $E$
has good reduction. Let the orders of $P_{1}$ and $Q_{1}$ modulo $p$ be
$r_{1}$, $r_{1}^{\prime}>3$, assume that $p\nmid r_{1}r_{1}^{\prime}$, and let
$P_{1}$, $P_{2}$, $P_{2}\pm P_{1}$, $Q_{1}$, $Q_{2}$, $Q_{2}\pm Q_{1}%
\not \equiv $ $\mathcal{O}$ $(\operatorname{mod}p)$. Assume further that
ord$_{p}\left(  W(0)\right)  $, ord$_{p}\left(  W(1)\right)  $, ord$_{p}%
\left(  \frac{W(2)}{W(0)}\right)  $, ord$_{p}\left(  \frac{W(3)}{W(1)}\right)
$ $\geq0$. Then there is a power $q=p^{N}$ such that for all $m\geq1$, the
limits
\[
\lim\limits_{i\rightarrow\infty}W(2mq^{i})\text{ and }\lim
\limits_{i\rightarrow\infty}W(2mq^{i}+1)\text{\ }%
\]
exist in $\mathbb{Z}_{p}$.
\end{theorem}

In \cite{Robinson}, Robinson used elementary methods to show that Somos $4$
and Somos $5$ sequences, whose initial terms and coefficients equal $1$, are
periodic modulo $n$ for every $n\in%
\mathbb{N}
$. He also made conjectures on modulo prime power periodicity properties of
these sequences. Swart \cite{CS} proved several of Robinson's conjectures, see
also \cite{VanK}. Hone \cite{AH4} considered the periodicity of Somos $5$
sequences modulo a prime, all of whose initial values are $p$-adic units and
coefficients equal to $1$.

In the following theorem, we use the periodicity properties of the sequence
$(\Psi_{n}(\mathbf{P}))_{n\geq0}$ to prove that the even and odd indexed
subsequences of a Somos $5$ sequence are purely periodic.

\begin{theorem}
With the notation of Theorem 1.5, assume that all the hypotheses of Theorem
1.5 hold, with the stronger conditions ord$_{p}\left(  \frac{W(2)}%
{W(0)}\right)  $, ord$_{p}\left(  \frac{W(3)}{W(1)}\right)  $ $=0$. Let the
orders of $P_{1}$ and $Q_{1}$ modulo $p$ be $r_{1}$, $r_{1}^{\prime}>3$,
respectively. Assume that $[r_{1}]P_{1}\neq\mathcal{O~}$and $[r_{1}^{\prime
}]Q_{1}\neq\mathcal{O}$, and let $p^{v}$ and $p^{w}$ be the highest powers of
$p$ such that $[r_{1}]P_{1}\equiv\mathcal{O~}(\operatorname{mod}$ $p^{v})$ and
$[r_{1}^{\prime}]Q_{1}\equiv\mathcal{O}$ $(\operatorname{mod}$ $p^{w})$,
respectively. Then, for all $l\geq1$, the even and odd indexed subsequences of
the Somos $5$ sequence $(W(n))_{n\geq0}$ are purely periodic modulo $p^{l}$
with periods $r_{l}t_{l}$ and $r_{l}^{\prime}t_{l}^{\prime}$ respectively, for
some positive integers $t_{l}$ and $t_{l}^{\prime}$ dividing $p^{l-1}(p-1)$,
where\ $r_{l}=p^{l-v}r_{1}$ for $l\geq v$ and $r_{l}=r_{1}$ for $l<v$,
$r_{l}^{\prime}=p^{l-w}r_{1}^{\prime}$ for $l\geq w $ and $r_{l}^{\prime
}=r_{1}^{\prime}$ for $l<w$.
\end{theorem}

\section{Translated division polynomials and their analytic representation}

This section recalls the analytic representation and algebraic definition of
the translated division polynomials. We then state the recurrence and
coordinate identities used in the periodicity arguments.

We first recall the analytic representation of translated division
polynomials: Let $E/%
\mathbb{C}
$ be an elliptic curve and consider the complex uniformization $\Phi:%
\mathbb{C}
/L\longrightarrow E(%
\mathbb{C}
)$ with a lattice $L\subset%
\mathbb{C}
$. Then the isomorphism $\Phi$ induces an isomorphism
\[
\Phi^{\ast}:(%
\mathbb{C}
/L)^{2}\longrightarrow E(%
\mathbb{C}
)^{2}%
\]
defined by $\Phi^{\ast}(\mathbf{z})=$ $(\Phi(z_{1}),\Phi(z_{2}))=\mathbf{P}$,
where $\mathbf{P}=(P_{1},P_{2})$ is a point such that $P_{1}$, $P_{2}$
$\neq\mathcal{O}$ and $P_{2}\pm P_{1}\neq\mathcal{O}$. Fix a point
$\mathbf{z}=(z_{1},z_{2})\in%
\mathbb{C}
^{2}$ with $z_{i}$, $z_{i}+z_{j}\notin L$. For $n\in%
\mathbb{Z}
$, Stange \cite{SK} (see also \cite{SKT}) defined a function on $%
\mathbb{C}
^{2}$
\begin{equation}
\Omega_{n}(\mathbf{z})=\Omega_{n}(\mathbf{z;~}L)=\frac{\sigma(nz_{1}+z_{2}%
)}{{\small \sigma(}z_{1},L{\small )}^{n{}^{2}-n}{\small \sigma(}z_{1}%
+z_{2},L{\small )}^{n}\sigma{\small (}z_{2}{\small )^{1-n}}}\label{7}%
\end{equation}
where $\sigma(z)=\sigma(z;~L)$ is the Weierstrass $\sigma$-function associated
to the lattice $L$. Then the map
\begin{equation}
\Psi(\mathbf{P};E):%
\mathbb{Z}
\rightarrow%
\mathbb{C}
\text{, }n\text{\textbf{\ }}\mathbf{\mapsto}\text{ }\Omega_{n}(\mathbf{z}%
)\label{5}%
\end{equation}
gives the sequence of \textit{translated division polynomial} \textit{values}
\textit{associated to} $E$ and $\mathbf{P}$ over $%
\mathbb{C}
$; see \cite[Theorem 3.7]{SK} and \cite[p.197]{AK} for the more general
elliptic net construction.

We now recall the algebraic definition and evaluation of translated division
polynomials. Stange's algebraic construction defines corresponding rational
functions over any field; see \cite[Section 4]{SK} and \cite[Theorem 2.3]{AK}.
We present the rank-two case needed here. Let $S^{\ast}=%
\mathbb{Z}
\lbrack\alpha_{1},\alpha_{2},\alpha_{3},\alpha_{4},\alpha_{6}]$ and let%
\begin{equation}
R^{\ast}=S^{\ast}[x_{i},y_{i}][(x_{1}-x_{2})^{-1}]/\left\langle f^{\ast}%
(x_{i},y_{i})\right\rangle \text{, }i=1,2\text{,}%
\end{equation}
where
\[
f^{\ast}(x,y)=y^{2}+\alpha_{1}xy+\alpha_{3}y-x^{3}-\alpha_{2}x^{2}-\alpha
_{4}x-\alpha_{6}\text{.}%
\]
Then there exists a morphism $\varphi_{_{\mathbf{P},E}}:R^{\ast}\rightarrow$
$K$ such that $\varphi_{_{\mathbf{P},E}}(\alpha_{r})=a_{r}$, and
$(\varphi_{_{\mathbf{P},E}}(x_{i}),\varphi_{_{\mathbf{P},E}}(y_{i}))=P_{i}$,
$i=1$, $2$ for every elliptic curve $E/K$ defined by the polynomial (\ref{1}),
and $\mathbf{P}=(P_{1},P_{2})\in$ $E(K)^{2}$ with $P_{1}$, $P_{2}$
$\neq\mathcal{O}$ and $P_{2}\pm P_{1}\neq\mathcal{O}$. Moreover, there exists
$\Psi_{n}^{\ast}\in R^{\ast}$ such that $\Psi^{\ast}:n\mapsto\Psi_{n}^{\ast}$
defines a translated division polynomial for each $n\in%
\mathbb{Z}
$, and $\varphi_{_{\mathbf{P},E}}\circ\Psi^{\ast}=\Psi(\mathbf{P};E)$ for any
elliptic curve $E/%
\mathbb{C}
$ and $\mathbf{P}=(P_{1},P_{2})\in$ $E(%
\mathbb{C}
)^{2}$ with $P_{1}$, $P_{2}$ $\neq\mathcal{O}$ and $P_{2}\pm P_{1}%
\neq\mathcal{O}$. Furthermore there is a map $\varphi_{_{E}}:S^{\ast
}\rightarrow$ $K$, such that $\varphi_{_{E}}(\alpha_{s})=a_{s}$, and this map
induces a map%
\[
\ (\varphi_{_{E}})_{\ast}:R^{\ast}\rightarrow R=K[x_{i},y_{i}][(x_{1}%
-x_{2})^{-1}]/\left\langle f(x_{i},y_{i})\right\rangle \text{, }i=1,2.
\]
For each $n\in%
\mathbb{Z}
$, define $\Psi_{n}=(\varphi_{_{E}})_{\ast}(\Psi_{n}^{\ast})\in R$. The
function $\Psi_{n}$ is called the $n$-\textit{translated division polynomial
associated to }$E$. Furthermore, if $\mathbf{P}\in$ $E(K)^{2}$ with $P_{1}$,
$P_{2}$ $\neq\mathcal{O}$ and $P_{2}\pm P_{1}\neq\mathcal{O}$, then
\begin{equation}
\Psi(\mathbf{P};E):%
\mathbb{Z}
\rightarrow K\text{, }n\mapsto\Psi_{n}(\mathbf{P})\label{6}%
\end{equation}
is the sequence of values of the translated division polynomials associated to
$E$ and $\mathbf{P}$ over $K$. In particular, when $K=%
\mathbb{C}
$, this sequence agrees with the analytic sequence (\ref{5}). More precisely,
Stange \cite[Theorem 6.1.1]{SKT} (see also \cite[Theorem 4.1]{SK}) showed
that
\begin{equation}
\Psi_{n}(\mathbf{P})=\Omega_{n}(\mathbf{z})\label{10}%
\end{equation}
for all $n\in%
\mathbb{Z}
$, where $\mathbf{z}$ corresponds to $\mathbf{P}$ under the complex
uniformization and satisfies the evaluation conditions above.

Although the translated division polynomials are defined for all integer
indices, our periodicity and convergence results concern sequences indexed by
$n\geq0$. Additional values with negative indices are used in the recurrence
and coordinate identities below.

We next give the recurrence relation and coordinate formulas for translated
division polynomials: Let $x_{1}$, $y_{1}$ and $x_{2}$, $y_{2}$ be coordinates
of the points $P_{1}$ and $P_{2}$, respectively. We can compute $\Psi_{n}$ by
using the initial values as follows
\begin{align*}
\Psi_{-1}  & =x_{1}-x_{2}\\
\Psi_{0}  & =\Psi_{1}=1\text{,}\\
\Psi_{2}  & =2x_{1}+x_{2}-\left(  \frac{y_{2}-y_{1}}{x_{2}-x_{1}}\right)
^{2}-a_{1}\left(  \frac{y_{2}-y_{1}}{x_{2}-x_{1}}\right)  +a_{2}%
\end{align*}
and by the relation
\begin{equation}
\Psi_{m+2}\Psi_{m-2}=\beta_{1}^{2}\Psi_{m+1}\Psi_{m-1}-\beta_{2}\Psi_{m}%
^{2}\label{s}%
\end{equation}
where $\beta_{1}=2y_{1}+a_{1}x_{1}+a_{3}$ and $\beta_{2}=3x_{1}^{4}+b_{2}%
x_{1}^{3}+3b_{4}x_{1}^{2}+3b_{6}x_{1}+b_{8}$ (here $b_{i}$ are the usual
quantities, see \cite[Chapter III.I]{JS1}). We refer the reader to Proposition
6.1.4 and Theorem 7.1.1 of \cite{SKT} for the initial values. Furthermore, the
recursive formula (\ref{s}) can be derived by using Theorem 5.2.1 of
\cite{SKT}. Moreover, the values of $\Phi_{n}$ and $\overline{\Omega}\,_{n}$
can be computed from the values of $\Psi_{n}$ as follows: let $E/K$ be an
elliptic curve given by the polynomial (\ref{1}), let $\mathbf{P}\in$
$E(K)^{2}$ be a point with $P_{1}$, $P_{2}$ $\neq\mathcal{O}$ and $P_{2}\pm
P_{1}\neq\mathcal{O}$, let $n\in$ $\mathbb{Z}$, and let $F$ be the field of
fractions of $R$. Then there exist rational functions $X_{n}$, $Y_{n}$ $\in F$
such that%
\[
\lbrack n]P_{1}+P_{2}=(X_{n}\text{, }Y_{n})\text{.}%
\]
In \cite[Lemma 2.5]{AK}, the authors used a theorem of Stange \cite[Lemma
4.2]{SK} to prove that
\begin{equation}
X_{n}=\frac{x(P_{1})\Psi_{n}^{2}-\Psi_{n+1}\Psi_{n-1}}{\Psi_{n}^{2}%
}\label{3.4}%
\end{equation}
for all $n\in$ $\mathbb{Z}$. Thus for any $n\in$ $\mathbb{Z}$ there exists
$\Phi_{n}\in R$ such that%
\[
X_{n}=\frac{\Phi_{n}}{\Psi_{n}^{2}}%
\]
and in particular,%
\begin{equation}
\Phi_{n}=x(P_{1})\Psi_{n}^{2}-\Psi_{n+1}\Psi_{n-1}\label{3.1}%
\end{equation}
for all $n$. In fact, one can use elliptic function theory (see \cite[Chapter
2]{Sc} for more details) to show that for any $n\in$ $\mathbb{Z}$ there exists
$\overline{\Omega}\,_{n}\in R$ such that%
\[
Y_{n}=\frac{\overline{\Omega}_{n}}{\Psi_{n}^{3}}%
\]
and in particular,%
\begin{equation}
\overline{\Omega}\,_{n}=(\Psi_{n-1}^{2}\Psi_{n+2}-\Psi_{n-2}\Psi_{n+1}%
^{2}-\beta_{1}\Psi_{n}(a_{1}\Phi_{n}+a_{3}\Psi_{n}^{2}))(2\beta_{1}%
)^{-1}\label{3.2}%
\end{equation}
for all $n$, where $\beta_{1}$ is as above and $2\beta_{1}\neq0$. Notice that
$\Omega_{n}(\mathbf{z})$ is the analytic representative of $\Psi
_{n}(\mathbf{P})$, whereas $\overline{\Omega}\,_{n}(\mathbf{P})$ is the
numerator appearing in the $y$-coordinate formula (\ref{2}). One can show that
$\overline{\Omega}\,_{n}\in R$, by setting
\[
\overline{\Omega}\,_{n}=Y_{n}\Psi_{n}^{3}\in\text{Frac}(R)\text{.}%
\]
Substituting $X_{n}=\Phi_{n}/\Psi_{n}^{2}$ and $Y_{n}=\overline{\Omega}%
\,_{n}/\Psi_{n}^{3}$ into the Weierstrass equation (\ref{1}), and then
clearing denominators gives
\begin{equation}
\overline{\Omega}\,_{n}^{2}+(a_{1}\Phi_{n}\Psi_{n}+a_{3}\Psi_{n}^{3}%
)\overline{\Omega}\,_{n}=\Phi_{n}^{3}+a_{2}\Phi_{n}^{2}\Psi_{n}^{2}+a_{4}%
\Phi_{n}\Psi_{n}^{4}+a_{6}\Psi_{n}^{6}\text{.}\label{3.3}%
\end{equation}
This is a monic equation for $\overline{\Omega}\,_{n}$ with coefficients in
$R$, so $\overline{\Omega}\,_{n}$ is integral over $R$. The ring $R$ is the
coordinate ring of the open subset $x_{1}\neq x_{2}$ of $(E\backslash
\{\mathcal{O}\})^{2}$. Since $E$ is smooth, this open subset is smooth, and
$R$ is integrally closed in its fraction field. Since $\overline{\Omega}%
\,_{n}\in$ Frac$(R)$, it follows that $\overline{\Omega}\,_{n}\in R$.

\begin{remark}
In Theorems 1.3 and 1.4, we assume that $p$ is an odd prime, and that the
chosen Weierstrass equation has coefficients in $\mathbb{Z}_{p}$ and good
reduction. If $P_{1},P_{2},P_{1}\pm P_{2}\not \equiv \mathcal{O~}%
(\operatorname{mod}p) $, then the coordinates of $P_{1}$ and $P_{2}$ lie in
$\mathbb{Z}_{p}$, and $x(P_{1})-x(P_{2})$ is a $p$-adic unit. The algebraic
construction therefore gives $\Psi_{n}(\mathbf{P})\in\mathbb{Z}_{p}$ for all
$n\in$ $\mathbb{N}$, see also \cite[Proposition 1.6]{AK}. Equation (\ref{3.1})
then gives $\Phi_{n}(\mathbf{P})\in\mathbb{Z}_{p}$ for all $n\in$ $\mathbb{N}%
$. On the other hand, if the order of the point $P_{1}$ modulo $p$ is $r>3$,
both $\beta_{1}$ and $\beta_{2}$ are nonzero modulo $p$. In particular,
$2\beta_{1}$ is a unit, so (\ref{3.2}) also gives $\overline{\Omega}%
\,_{n}(\mathbf{P})$ $\in\mathbb{Z}_{p}$ for all $n\in$ $\mathbb{N}$. This
explains the role of $r>3$ in the recurrence and coordinate formulas. When the
order of the point $P_{1}$ modulo $p$ is $2$, formula (\ref{3.2}) cannot be
used directly modulo $p$, since its denominator vanishes.
\end{remark}

\section{Periodicity of the Sequence $(\Psi_{n}(\mathbf{P}))_{n\geq0}$ Over
Finite Fields}

We first establish a symmetry formula for $(\Psi_{n}(\mathbf{P}))$, from which
pure periodicity follows. For related symmetry and periodicity results, see
Swart \cite[Theorems 9.3.2 and 9.4.1]{CS}, Stange \cite[Theorem 10.2.3]{SKT},
and Akbary et al. \cite[Theorem 1.13]{AK}. Our proof follows the approach of
Silverman \cite[Theorem 8]{JS3}: we lift $P_{1}$ to a torsion point in
characteristic zero, use the quasi-periodicity of the Weierstrass sigma
function, and then reduce the resulting identity to the residue field.

Let $R$ be a complete local ring of characteristic zero with residue field
${\mathbb{F}}$ of characteristic $p>0$, let $K$ be the field of fractions of
$R $, and let $\mathfrak{p}$ be the maximal ideal of $R$. To prove our result,
we need to know the behavior of zeros of the sequence $(\Psi_{n}%
(\mathbf{P})(\text{mod }\mathfrak{p}))_{n\geq0}$. For the sequence of values
$(\Psi_{n}(\mathbf{P}))_{n\geq0}$, we have%
\begin{equation}
\Psi_{n}(\mathbf{P})\equiv0~(\text{mod }\mathfrak{p})\Leftrightarrow\lbrack
n]P_{1}+P_{2}\equiv\mathcal{O~}(\text{mod }\mathfrak{p})\text{,}\label{y1}%
\end{equation}
see \cite[Theorem 7.6.1]{CS}, for more details.

\begin{theorem}
Let ${\mathbb{F}}$ be a finite field, let $E/{\mathbb{F}}$ be an elliptic
curve, and let $\mathbf{P}=(P_{1},P_{2})\in$ $E({\mathbb{F}})^{2}$. Let the
order of $P_{1}$ be $r>3$ in $E({\mathbb{F}})$, let $P_{2}$, $P_{1}\pm
P_{2}\neq$ $\mathcal{O}$, and let $(\Psi_{n}(\mathbf{P}))_{_{n\geq0}}$ be the
sequence of values of the translated division polynomials of $E$ at
$\mathbf{P}$. Then there exist units $a$, $b$ $\in{\mathbb{F}}^{\ast}$,
depending on $\mathbf{P}$, such that for all nonnegative integers $k$, $n$
\begin{equation}
\Psi_{kr+n}(\mathbf{P})=a^{kn}(a^{r})^{k^{2}/2}b^{k}{}\Psi_{n}(\mathbf{P}%
)\text{.}\label{01}%
\end{equation}
Furthermore, the units $a$ and $b$ satisfy
\begin{equation}
a=\frac{\Psi_{r+1}(\mathbf{P})}{\Psi_{r}(\mathbf{P})}\text{, \ \ }b=\frac
{\Psi_{r}(\mathbf{P})}{(a^{r})^{1/2}}\text{.}%
\end{equation}

\end{theorem}

\begin{proof}
We divide the proof into three steps.

\textit{Lifting and uniformization}: Let $R$ be a complete discrete valuation
ring of characteristic zero with residue field ${\mathbb{F}}$ of
characteristic $p>0$, let $K$ be the field of fractions of $R$, and let
$\mathfrak{p}$ be the maximal ideal of $R$, let
\[
\Gamma=\{n\in%
\mathbb{N}
~|~\Psi_{n}(\mathbf{P})=0\}\text{,}%
\]
and let $n\in%
\mathbb{N}
\backslash\Gamma$, since otherwise, both sides of the equation (\ref{01}) are
zero, and the formula is true for any choice of $a$ and $b$. Let
$\mathcal{E}/R$ be a lift of $E/{\mathbb{F}}$ given by a Weierstrass equation
whose reduction modulo $\mathfrak{p}$ is the Weierstrass equation of
$E/{\mathbb{F}}$. We can choose a lift $P_{2}^{\ast}$ of $P_{2}$ and lift
$P_{1} $ to a torsion point $P_{1}^{\ast}$ of exact order $r$, passing to a
finite extension if necessary. If ${\mathbb{~}}p\nmid r$, the torsion lift
$P_{1}^{\ast}$ is unique and can be obtained from any lift $Q_{1}$ of $P_{1}$
by%
\[
P_{1}^{\ast}=\lim_{\substack{i\rightarrow\infty\\p^{i}\equiv1~(\text{mod}%
~r)}}[p^{i}]Q_{1}\text{;}%
\]
see \cite[Proposition 10]{JS3}. The point $P_{2}^{\ast}$ is kept fixed in this
construction. In general, the torsion lift can be chosen over a finite
extension $K^{\prime}/K$, as explained in the proof of Theorem 8 of
\cite{JS3}. Writing $R^{\prime}$ for its ring of integers and $\mathfrak{p}%
^{\prime}$ for its maximal ideal, we obtain
\[
\mathbf{P}^{\ast}=(P_{1}^{\ast},P_{2}^{\ast})\in\mathcal{E}(R^{\prime}%
)^{2}\text{, }P_{1}^{\ast}\equiv P_{1}(\text{mod }\mathfrak{p}^{\prime
})\text{, and }P_{2}^{\ast}\equiv P_{2}(\text{mod }\mathfrak{p}^{\prime
})\text{,}%
\]
with $[r]P_{1}^{\ast}=\mathcal{O}$. Since $n\in%
\mathbb{N}
\backslash\Gamma$, reduction shows that $[n]P_{1}^{\ast}+P_{2}^{\ast}%
\neq\mathcal{O}$.

Now we take a subfield $K_{0}^{\prime}$ of $K^{\prime}$ which is small enough
to embed $K_{0}^{\prime}$ into $%
\mathbb{C}
$ but large enough such that $\mathbf{P}^{\ast}$ $\in$ $\mathcal{E}%
(K_{0}^{\prime})^{2}$. Then we have an embedding
\[
\mathcal{E}(K_{0}^{\prime})^{2}\subset\mathcal{E}(%
\mathbb{C}
)^{2}\overset{~~~{\small \Phi}^{\ast}}{{\large \leftarrow}}(%
\mathbb{C}
/L)^{2}%
\]
for some lattice $L$. Now, let $\Phi:$ $%
\mathbb{C}
/L\rightarrow\mathcal{E}(%
\mathbb{C}
)$, and let $\mathbf{P}^{\ast}=\Phi^{\ast}(\mathbf{z})=$ $(\Phi(z_{1}%
),\Phi(z_{2}))$ for some $\mathbf{z}\in(%
\mathbb{C}
/L)^{2}$ under this identification.

\textit{The sigma-function calculation}: Then by (\ref{10}),%
\begin{equation}
\Psi_{n}(\Phi^{\ast}(\mathbf{z}))=\Omega_{n}(\mathbf{z})
\end{equation}
for all $n\in%
\mathbb{Z}
$ and for all $\mathbf{z}\in%
\mathbb{C}
^{2}$ satisfying the evaluation conditions in Section 2. Then by (\ref{7}),%
\begin{equation}
\Psi_{kr+n}(\Phi(z_{1}),\Phi(z_{2}))=\frac{\sigma(krz_{1}+nz_{1}+z_{2}%
)}{{\small \sigma(}z_{1}{\small )}^{(kr+n){}^{2}-(kr+n)}{\small \sigma(}%
z_{1}+z_{2}{\small )}^{kr+n}{\small \sigma(}z_{2}{\small )}^{1-kr-n}}%
\end{equation}
for all $z_{i}$ $\in%
\mathbb{C}
$ with $z_{i}$ and $z_{i}+$ $z_{j}\notin L$. Now multiplying and then dividing
the last equation by $\Psi_{n}$ yields%
\begin{align}
\Psi_{kr+n}(\Phi(z_{1}),\Phi(z_{2}))  & =\frac{\sigma(krz_{1}+nz_{1}+z_{2}%
)}{\sigma(nz_{1}+z_{2})}{\small (\sigma(}z_{1}{\small )}^{kr-2krn-k^{2}r{}%
^{2}}{\small )}\nonumber\\
& \times{\small (\sigma(}z_{1}+z_{2}{\small )}^{-kr}{\small \sigma(}%
z_{2}{\small )}^{kr})\Psi_{n}(\Phi(z_{1}),\Phi(z_{2}))\label{r1}%
\end{align}
where $nz_{1}+z_{2}\notin L$.

By assumption, $z_{1}\in%
\mathbb{C}
/L$ is a point such that $rz_{1}\in L$, thus $krz_{1}\in L$. Then we can apply
the transformation formula%
\[
\sigma(z+\omega)=\delta(\omega)e^{\eta(\omega)(z+\omega/2)}\sigma
(z)\ \ \text{for all \ }z\in%
\mathbb{C}
\text{ and }\omega\in L,
\]
where $\eta:L\rightarrow%
\mathbb{C}
$ is the quasiperiodic map associated to $L$, and $\delta$ is defined by%
\[
\delta:L\rightarrow\{\pm1\}\text{,}\ \ \ \ \delta(\omega)=\left\{
\begin{array}
[c]{cc}%
1 & \text{if }\omega/2\in L\text{,}\\
-1 & \text{if }\omega/2\notin L\text{.}%
\end{array}
\right.
\]
Now using the fact that $\eta$ is a homomorphism and the identity
$\delta(k\omega)=\delta(\omega)^{k}$, applying the transformation formula with
$z=nz_{1}+z_{2}$ and $\omega=krz_{1}$, we have%
\begin{align}
\frac{\sigma(z+\omega)}{\sigma(z)}  & =\delta(krz_{1})e^{\eta(krz_{1}%
)(nz_{1}+z_{2}+\frac{1}{2}krz_{1})}\nonumber\\
& =\delta(rz_{1})^{k}e^{\eta(rz_{1})knz_{1}}e^{\eta(rz_{1})kz_{2}}%
e^{\eta(rz_{1})rz_{1}k^{2}/2}\text{.}\label{r2}%
\end{align}
This calculation is valid whenever $z\notin L$, in particular for
$n\notin\Gamma$.

Now substituting (\ref{r2}) into (\ref{r1}) we have
\[
\Psi_{kr+n}(\Phi(z_{1}),\Phi(z_{2}))=\alpha^{kn}(\alpha^{r})^{k^{2}/2}%
\beta^{k}\Psi_{n}(\Phi(z_{1}),\Phi(z_{2}))
\]
for all $k\geq0$ and $n\in%
\mathbb{N}
\backslash\Gamma$ where
\[
\alpha=e^{\eta(rz_{1})z_{1}}{\small \sigma(z}_{1}{\small )}^{-2r}\text{ \ and
\ }\beta=\delta(rz_{1})e^{\eta(rz_{1})z_{2}}{\small \sigma(z}_{1}%
{\small )}^{r}{\small \sigma(z}_{2}{\small )}^{r}{\small \sigma(z}%
_{1}+{\small z}_{2}{\small )}^{-r}%
\]
and so
\[
(\alpha^{r})^{1/2}=e^{\eta(rz_{1})rz_{1}/2}{\small \sigma(z}_{1}%
{\small )}^{-r^{2}}\text{.}%
\]
It follows that
\begin{equation}
\Psi_{kr+n}(\mathbf{P}^{\ast})=\alpha^{kn}(\alpha^{r})^{k^{2}/2}\beta^{k}%
\Psi_{n}(\mathbf{P}^{\ast})\label{81}%
\end{equation}
for all $k$, $n\geq0$, since $\mathbf{P}^{\ast}=(\Phi(z_{1}),\Phi(z_{2})) $.
The identity holds for all $n\geq0$. Indeed, if $nz_{1}+z_{2}\notin L$, the
preceding calculation applies. If $nz_{1}+z_{2}\in L$, then $krz_{1}%
+nz_{1}+z_{2}\in L$, since $rz_{1}\in L$, and both sides of (\ref{81}) vanish.

\textit{Reduction to the original residue field}: We now show that the
constants are units and that their reductions lie in the original field
${\mathbb{F}}$. Setting $k=1$ in (\ref{81}) we have%
\begin{equation}
\Psi_{r+n}(\mathbf{P}^{\ast})=\alpha^{n}(\alpha^{r})^{1/2}\beta\Psi
_{n}(\mathbf{P}^{\ast})\text{.}\label{8}%
\end{equation}
Now substituting $n=0$ and then $n\mathbf{~}=1$ into (\ref{8}) we obtain
\[
\Psi_{r}(\mathbf{P}^{\ast})=(\alpha^{r})^{1/2}\beta\text{,}%
\]
and
\[
\Psi_{r+1}(\mathbf{P}^{\ast})=\alpha(\alpha^{r})^{1/2}\beta\text{,}%
\]
since $\Psi_{0}(\mathbf{P}^{\ast})=\Psi_{1}(\mathbf{P}^{\ast})=1$. Thus
\[
\alpha=\frac{\Psi_{r+1}(\mathbf{P}^{\ast})}{\Psi_{r}(\mathbf{P}^{\ast})}\in
K(\mathbf{P}^{\ast})\subset K^{\prime}\text{.}%
\]
After replacing $K^{\prime}$ by a finite extension if necessary, we may assume
that the chosen square root $(\alpha^{r})^{1/2}$ belongs to $K^{\prime}$.
Consequently,
\begin{equation}
\beta=\frac{\Psi_{r}(\mathbf{P}^{\ast})}{(\alpha^{r})^{1/2}}\in K^{\prime
}\text{.}%
\end{equation}
By (\ref{y1}), both $\Psi_{r}(\mathbf{P}^{\ast})$ and $\Psi_{r+1}%
(\mathbf{P}^{\ast})$ are $\mathfrak{p}^{\prime}$-units, since their
corresponding points reduce to $P_{2}\neq\mathcal{O}$ and $P_{1}+P_{2}%
\neq\mathcal{O}$, respectively. Hence $\alpha$, $(\alpha^{r})^{1/2}$, and
$\beta$ are all $\mathfrak{p}^{\prime}$-units in $K^{\prime}$. It follows that
there exist $\mathfrak{p}^{\prime}$-units $\alpha$ and $\beta$ in $K^{\prime}$
such that
\begin{equation}
\Psi_{kr+n}(\mathbf{P}^{\ast})=\alpha^{kn}(\alpha^{r})^{k^{2}/2}\beta^{k}%
\Psi_{n}(\mathbf{P}^{\ast})\label{9}%
\end{equation}
for all $k$, $n\geq0$. Note that we reduce formula (\ref{9}) modulo
$\mathfrak{p}^{\prime}$. Now using the fact that $\mathbf{P}^{\ast}%
\equiv\mathbf{P}~(\text{mod }\mathfrak{p}^{\prime})$ we have%
\begin{equation}
\Psi_{kr+n}(\mathbf{P})=a^{kn}(a^{r})^{k^{2}/2}b^{k}\Psi_{n}(\mathbf{P})
\end{equation}
for all $k$, $n\geq0$, where $a$ and $b$ are elements of the residue field
$F^{\prime}$of $K^{\prime}$. Here $(a^{r})^{1/2}$ denotes the reduction of the
chosen square root $(\alpha^{r})^{1/2}$, and $(a^{r})^{k^{2}/2}$ denotes its
$k^{2}$-th power. On the other hand, putting $k=1$ in the above equation we
have%
\begin{equation}
\Psi_{r+n}(\mathbf{P})=a^{n}(a^{r})^{1/2}b\Psi_{n}(\mathbf{P})\text{.}%
\label{3.13}%
\end{equation}
Setting $n=0$ and then $n=1$ in (\ref{3.13}) and using $\Psi_{0}%
(\mathbf{P})=\Psi_{1}(\mathbf{P})=1$ gives
\[
a=\frac{\Psi_{r+1}(\mathbf{P})}{\Psi_{r}(\mathbf{P})}\text{ }\in{\mathbb{F}%
}^{\ast}\text{.}%
\]
The ordinary division polynomial calculation in the proof of Theorem 8 of
\cite{JS3}, using the same lift and uniformization, has the same
mixed-exponent coefficient $\alpha=e^{\eta(rz_{1})z_{1}}{\small \sigma(z}%
_{1}{\small )}^{-2r}$. Thus it yields a constant $u$ $\in{\mathbb{F}}^{\ast}$
satisfying%
\begin{equation}
\psi_{kr+i}(P_{1})=a^{ki}u^{k^{2}}\psi_{i}(P_{1})\label{315}%
\end{equation}
for all $k$, $i\geq0$. Setting $k=1$ and $i=1$ in (\ref{315}), we obtain
$\psi_{r+1}(P_{1})=au$, since $\psi_{1}(P_{1})=1$. Taking $k=1$ and $i=r+1$ in
(\ref{315}) gives $\psi_{2r+1}(P_{1})=a^{r+1}u\psi_{r+1}(P_{1})=a^{r+2}u^{2}$.
On the other hand taking $k=2$ and $i=1$ in (\ref{315}) yields $\psi
_{2r+1}(P_{1})=a^{2}u^{4}$. Hence $u^{2}=a^{r}$. Thus the chosen square root
$(a^{r})^{1/2}$ equals $u$ or $-u$ and therefore belongs to ${\mathbb{F}%
}^{\ast}$. Consequently,
\[
b=\frac{\Psi_{r}(\mathbf{P})}{(a^{r})^{1/2}}\text{ }\in{\mathbb{F}}^{\ast
}\text{,}%
\]
which completes the proof.
\end{proof}

As a consequence of the above theorem, one can immediately obtain the
following corollary. See \cite[Theorem 9.4.1]{CS} for a proof.

\begin{corollary}
With the assumptions and notation as in Theorem 3.1, the sequence $(\Psi
_{n}(\mathbf{P}))_{n\geq0}$ is purely periodic with period $rt$ for some
integer $t$ dividing $q-1$, where $q$ is the order of the field ${\mathbb{F}}
$.
\end{corollary}

This completes the proof of Theorem 1.1 for the sequence $(\Psi_{n}%
(\mathbf{P}))$.

\section{Periodicity of the Sequences $(\Phi_{n}(\mathbf{P}))_{n\geq0}$ and
$(\overline{\Omega}_{n}(\mathbf{P}))_{n\geq0}$ Over Finite Fields}

In this section, we give the proof of Theorem 1.1 for the sequences $(\Phi
_{n}(\mathbf{P}))_{n\geq0}$ and $(\overline{\Omega}_{n}(\mathbf{P}))_{n\geq0}$
of values of the translated division polynomials of an elliptic curve $E$ at
$\mathbf{P}=(P_{1},P_{2})\in E(K)^{2}$. We also establish the periodicity
properties of these sequences and prove that the periods of these sequences
can be given by using the order of the point $P_{1}$, similar to the sequence
$(\Psi_{n}(\mathbf{P}))_{n\geq0}$.

\begin{theorem}
With the assumptions and notation of Theorem 3.1, let $(\Phi_{n}%
(\mathbf{P}))_{n\geq0}$ and $(\overline{\Omega}_{n}(\mathbf{P}))_{n\geq0}$ be
sequences of values of the translated division polynomials of $E$ at
$\mathbf{P}$. Then there exist units $a$, $b$ $\in{\mathbb{F}}^{\ast}$,
depending on $\mathbf{P}$, such that%
\begin{equation}
\Phi_{kr+n}(\mathbf{P})=a^{2kn}(a^{r})^{k^{2}}b^{2k}{}\Phi_{n}(\mathbf{P}%
)\text{,}\label{6.1}%
\end{equation}
and%
\begin{equation}
\overline{\Omega}_{kr+n}(\mathbf{P})=a^{3kn}(a^{r})^{3k^{2}/2}b^{3k}%
{}\overline{\Omega}\,_{n}(\mathbf{P})\label{6.2}%
\end{equation}
for all nonnegative integers $k$, $n$, where $(a^{r})^{1/2}$ is chosen as in
the proof of Theorem 3.1.
\end{theorem}

\begin{proof}
If we replace $n$ by $kr+n$ in (\ref{3.1}) and then use equation (\ref{01}),
we derive the formula (\ref{6.1}). Similarly, if char(${\mathbb{F}}$) $\neq2
$, replacing $n$ by $kr+n$ in (\ref{3.2}) and then using equation (\ref{01}),
we obtain the formula (\ref{6.2}), since $2\beta_{1}\neq0$. Suppose that
char(${\mathbb{F}}$) $=2$. If $\Psi_{n}(\mathbf{P})\neq0$, the second identity
follows by comparing the $y$-coordinates of the equal points $[kr+n]P_{1}%
+P_{2}$ and $[n]P_{1}+P_{2}$. Otherwise, $\Psi_{kr+n}(\mathbf{P})=0$ as well.
Equation (\ref{3.3}) gives $\overline{\Omega}_{kr+n}(\mathbf{P})^{2}%
=\Phi_{kr+n}(\mathbf{P})^{3}$ and so $\overline{\Omega}_{kr+n}(\mathbf{P}%
)^{2}=a^{6kn}(a^{r})^{3k^{2}}b^{6k}{}\overline{\Omega}_{n}(\mathbf{P})^{2}$.
Since char(${\mathbb{F}}$) $=2$, the preceding equality implies the desired result.
\end{proof}

As an immediate consequence of the above result, we deduce the periodicity of
the sequences $(\Phi_{n}(\mathbf{P}))_{n\geq0}$ and $(\overline{\Omega}{}%
_{n}(\mathbf{P}))_{n\geq0}$.

\begin{corollary}
With the assumptions and notation as in Theorem 4.1, the sequences $(\Phi
_{n}(\mathbf{P}))_{n\geq0}$ and $(\overline{\Omega}_{n}(\mathbf{P}))_{n\geq0}$
are purely periodic with periods $rt^{\prime}$ and $rt^{\prime\prime}%
$respectively, for some integers $t^{\prime}$ and $t^{\prime\prime}$ dividing
$q-1$, where $q$ is the order of the field.
\end{corollary}

\begin{proof}
Let $a$ and $b$ be as in Theorem 3.1 and let $t^{\prime}\geq1$ be the smallest
integer such that%
\begin{equation}
a^{2t^{\prime}}=1\text{ and }a^{rt^{\prime2}}b^{2t^{\prime}}=1\text{.}%
\end{equation}
As $t^{\prime}$ divides the least common multiple of the orders of $a$ and $b$
in ${\mathbb{F}}^{\ast}$, $t^{\prime}$ divides $q-1$. The equation (\ref{6.1})
implies that
\[
\Phi_{rt^{\prime}+n}(\mathbf{P})=a^{2t^{\prime}n}(a^{r})^{t^{\prime2}%
}b^{2t^{\prime}}{}\Phi_{n}(\mathbf{P})=\Phi_{n}(\mathbf{P})
\]
for all $n\geq0$. It follows that the sequence $(\Phi_{n}(\mathbf{P}%
))_{n\geq0}$ is periodic and $rt^{\prime}$ is a period. The proof of the
periodicity of the sequence $(\overline{\Omega}_{n}(\mathbf{P}))_{n\geq0}$ is similar.
\end{proof}

Corollary 4.2 shows that the sequences $(\Phi_{n}(\mathbf{P}))_{n\geq0}$ and
$(\overline{\Omega}{}{}_{n}(\mathbf{P}))_{n\geq0}$ are purely periodic, which
completes the proof of Theorem 1.1.

In the following theorem, we give explicit formulas for periods of the
sequences $(\Phi_{n}(\mathbf{P}))_{n\geq0}$ and $(\overline{\Omega}{}{}%
_{n}(\mathbf{P}))_{n\geq0}$. But first, we need a useful lemma.

\begin{lemma}
[{\cite[Lemma 11.1]{MW1}}]Let $p$ be an odd prime and $d$ be an integer such
that $\gcd(p,~d)=1$. Let $\delta$ be the least positive integer such that
$d^{\delta}\equiv1~(\operatorname{mod}p)$. If $\delta$ is odd, there exists no
integer $x$ such that the congruence $d^{x}\equiv-1~(\operatorname{mod}p)$ is
satisfied. But if $\delta$ is even, the last congruence is satisfied if and
only if $x$ is an odd multiple of $\delta/2$.
\end{lemma}

We can now give explicit expressions for the period factors appearing in the
preceding corollaries. To avoid repeating the same case distinction, define%
\begin{equation}
t=t(\varepsilon,~\kappa,~\lambda)=2^{\mu(\varepsilon,~\kappa,~\lambda
)}\text{lcm}[\varepsilon,~\kappa,~\lambda]\label{t1}%
\end{equation}
for positive integers $\varepsilon$, $\kappa$ and $\lambda$ where%
\[
\mu(\varepsilon,~\kappa,~\lambda)=\left\{
\begin{array}
[c]{cc}%
-1\text{,} &
\begin{array}
[c]{c}%
\begin{array}
[c]{l}%
\text{if }\varepsilon\text{ and }\lambda\text{ are both even, and }%
\lambda\text{ is divisible }\\
\text{by\ a higher power of }2\text{ than }\varepsilon\text{ and }%
\kappa\text{,}%
\end{array}
\\
\text{ \ \ \ or if }\varepsilon\text{ is odd, and ord}_{2}(\kappa
)=\text{ord}_{2}(\lambda)=1\text{, \ \ \ \ \ \ }%
\end{array}
\\
\text{ \ }0\text{,} &
\text{otherwise.\ \ \ \ \ \ \ \ \ \ \ \ \ \ \ \ \ \ \ \ \ \ \ \ \ \ \ \ \ \ \ \ \ \ \ \ \ \ \ \ \ }%
\end{array}
\text{ }\right.
\]

\begin{theorem}
Let $E/%
\mathbb{Q}
$ be an elliptic curve and let $\mathbf{P}=(P_{1},P_{2})$ $\in$ $E(%
\mathbb{Q}
)^{2}$ be a pair consisting of points of infinite order. Let $p$ be an odd
prime at which the chosen integral model of $E$ has good reduction, and let
$r>3$ be the order of $P_{1}$ modulo $p$. Assume that $P_{2},P_{1}\pm P_{2}$
$\not \equiv \mathcal{O~}(\operatorname{mod}p)$. Let $(\Phi_{n}(\mathbf{P}%
))_{n\geq0} $ and $(\overline{\Omega}{}{}{}_{n}(\mathbf{P}))_{n\geq0}$ be the
sequences of values of the translated division polynomials of $E$ at
$\mathbf{P}$. Let $a$ and $b$ be as in Theorem 3.1. Let $\varepsilon$,
$\kappa$, $\lambda$ and $\varepsilon^{\prime}$, $\kappa^{\prime}$,
$\lambda^{\prime}$ be the orders of $a^{2}$, $b^{2}$, $a^{r}$ and $a^{3}$,
$b^{3}$, $(a^{r}{})^{3/2}$ modulo $p$, respectively, where the square root is
chosen as in the proof of Theorem 3.1. Then, the sequences $(\Phi
_{n}(\mathbf{P})$ $\operatorname{mod}p)_{n\geq0}$ and $(\overline{\Omega}{}%
{}_{n}(\mathbf{P})$ $\operatorname{mod}p)_{n\geq0}$ are purely periodic with
periods $rt(\varepsilon,~\kappa,~\lambda)$ and $rt(\varepsilon^{\prime
},~\kappa^{\prime},~\lambda^{\prime})$ respectively, where $t$ is given by
(\ref{t1}).
\end{theorem}

\begin{proof}
We give the proof for $(\Phi_{n}(\mathbf{P})\operatorname{mod}p)_{n\geq0}$;
the proof for $(\overline{\Omega}{}{}_{n}(\mathbf{P})\text{ }%
\operatorname{mod}p)_{n\geq0}$ is similar. Let $a$ and $b$ be as in Theorem
3.1, and let $t\geq1$ be the smallest integer such that%
\begin{equation}
a^{2t}\equiv1\text{ \ \ and \ }a^{rt^{2}}b^{2t}\equiv1\text{ }(\text{mod
}p)\text{.}\label{a4}%
\end{equation}
Since $a^{2t}\equiv1$ $(\text{mod }p)$, we have $a^{2rt}\equiv1$ $(\text{mod
}p)$, and so $a^{-2rt^{2}}\equiv b^{4t}$ $\equiv1~(\text{mod }p)$. It follows
that $b^{4t}$ $=(b^{2t})^{2}$ $\equiv1$ $(\text{mod }p)$. Thus $b^{2t}$
$\equiv1$ or $-1$ $(\text{mod }p)$.

Let $\varepsilon$, $\kappa$, and $\lambda$ denote the orders of $a^{2}$,
$b^{2}$, and $a^{r}$ modulo $p$, respectively. Since $a^{2rt}\equiv1$
$(\text{mod }p)$, we have $a^{rt}\equiv\pm1$ $(\text{mod }p)$. We distinguish
the following three cases.

\textbf{Case 1.} $b^{2t}$ $\equiv1$ \ and\ $a^{rt}\equiv1$ $(\text{mod }p)$.
Let $s=$ $\text{lcm}[\varepsilon,~\kappa,~\lambda]$. Then (\ref{a4}) implies
that $\varepsilon~|~t$, $\kappa~|~t$, and $\lambda~|~t$, and hence $s~|~t$. On
the other hand,
\[
a^{2s}\equiv b^{2s}\equiv a^{rs}\equiv1\text{ }(\operatorname{mod}p).
\]
Thus
\begin{equation}
a^{2s}\equiv1\text{ and \ }a^{r}{}^{s^{2}}b^{2s}\equiv1\text{ }%
(\operatorname{mod}p).\label{a7}%
\end{equation}
By the minimality of $t$ in (\ref{a4}), we have $t\leq s$. Since $s|t$, it
follows that $s=t$.

\textbf{Case 2.} $b^{2t}$ $\equiv1$ \ and\ $a^{rt}\equiv-1$ $(\text{mod }p) $.
By (\ref{a4}), we have $(-1)^{t}\equiv1$ $(\text{mod }p)$, so $t$ must be
even. On the other hand, by Lemma 4.1, $\lambda$ is even, and $t$ is an odd
multiple of $\lambda/2$, i.e., $\lambda$ is divisible by $2^{x}$, with
$x\geq2$, since $t$ is even. Now let $s=$ $\frac{1}{2}\text{lcm}%
[\varepsilon,~\kappa,~\lambda]$. Then $s$ must be even since $\lambda$ is
divisible by $2^{x}$, with $x\geq2$. On the other hand, since $\varepsilon
~|~t$, $\kappa~|~t$, and $\lambda~|~2t$, we have lcm$[\varepsilon
,~\kappa,~\lambda]~|~2t$, and hence $s~|~t$.

We also note that $a^{2s}\equiv1~(\text{mod }p)$ if and only if $s$ is a
multiple of $\varepsilon$ if and only if $\kappa$ or$~\lambda$ is divisible by
a higher power of $2$ than $\varepsilon$. Similarly, $b^{2s}\equiv1~(\text{mod
}p)$ if and only if $s$ is a multiple of $\kappa$ if and only if $\varepsilon$
or$~\lambda$ is divisible by a higher power of $2$ than $\kappa$. It follows
that $a^{2s}\equiv b^{2s}\equiv1$ $(\text{mod }p)$ if and only if $\lambda$ is
divisible by\ a higher power of $2$ than $\varepsilon$ and $\kappa$, and also,
$(a^{r})^{s}\equiv-1~(\text{mod }p)$ if and only if $\lambda$ is even, and $s$
is an odd multiple of $\lambda/2$ by Lemma 4.1.

Therefore if $\varepsilon$ and $\lambda$ are both even and $\lambda$ is
divisible by a higher power of $2$ than $\varepsilon$ and $\kappa$, then
$a^{2s}\equiv b^{2s}\equiv1$ $(\text{mod }p)$, and $a^{r}{}^{s}\equiv
-1~(\text{mod }p)$. Thus
\begin{equation}
a^{2s}\equiv1\text{, and \ }a^{r}{}^{s^{2}}b^{2s}\equiv1\text{ }(\text{mod }p)
\end{equation}
since $s$ is even. By the minimality of $t$ in (\ref{a4}), we have $t\leq s$.
Since $s|t$, it follows that $s=t$.

\textbf{Case 3. }$b^{2t}$ $\equiv-1$ $(\text{mod }p)$. Then $t$ must be odd
and $a^{r}{}^{t^{2}}\equiv-1$ $(\text{mod }p)$ by (\ref{a4}). Thus $a^{r}%
{}^{t}\equiv-1$ $(\text{mod }p)$. It follows that $\kappa$ and $\lambda$ are
both even, and $t$ is an odd multiple of both $\kappa/2$ and $\lambda/2$ by
Lemma 4.1. On the other hand, $\varepsilon$ must be odd since $\varepsilon
~|~t$ by (\ref{a4}). But if $s=$ $\frac{1}{2}\text{lcm}[\varepsilon
,~\kappa,~\lambda]$, then $s~|~t$, and also $s$ must be odd, since $t$ is odd.
Thus $s$ must be an odd multiple of both $\kappa/2$ and $\lambda/2$, say
$s=\frac{\kappa}{2}m$ and $s=\frac{\lambda}{2}n$ for odd integers $m$ and $n$,
since $\varepsilon$ is odd. Furthermore, ord$_{2}(\kappa/2)=$ ord$_{2}%
(\lambda/2)=0$. Thus if $b^{2t}$ $\equiv-1$ and $a^{r}{}^{t^{2}}\equiv-1$
$(\text{mod }p)$, then we have $\varepsilon$ is odd, and also $\kappa$ and
$\lambda$ are both even and both divisible by exactly $2$.

On the other hand, if $\kappa$ and $\lambda$ are both even and both divisible
by exactly $2$, then we can write $\kappa=2u$, $\lambda=2v$, for odd integers
$u$ and $v$. Thus if $\varepsilon$ is odd, then $s=$ $\text{lcm}%
[\varepsilon,~\frac{\kappa}{2},~\frac{\lambda}{2}]=\text{lcm}[\varepsilon
,~u,~v]$. Since $s=\varepsilon k=uw=vz$, for some odd integers $k$, $w$, and
$z$, we can write $s=\varepsilon k$ $=\frac{\kappa}{2}w$ $=\frac{\lambda}{2}%
z$. Thus $s$ is an odd multiple of $\frac{\kappa}{2}$ and $\frac{\lambda}{2}$.
Hence%
\[
b^{2s}\equiv a^{rs}\equiv-1\text{ }(\text{mod }p)
\]
by Lemma 4.1. Therefore we have
\[
a^{2s}\equiv1\text{ and \ }a^{r}{}^{s^{2}}b^{2s}\equiv1\text{ }(\text{mod
}p)\text{,}%
\]
where the first congruence follows from the fact that $a^{2s}\equiv
1~(\text{mod }p)$ if and only if $\kappa$ or$~\lambda$ is divisible by a
higher power of $2$ than $\varepsilon$. By the minimality of $t$ in
(\ref{a4}), we have $t\leq s$. Since $s|t$, it follows that $s=t$.

These cases give precisely the expression for the period factor $t(\varepsilon
,~\kappa,~\lambda)$ in (\ref{t1}), completing the proof.
\end{proof}

\section{Periodicity modulo $p^{l}$}

We now consider the case of modulo prime powers $p^{l}$ for $l\geq2$.
Periodicity results for division polynomial sequences were obtained by Shipsey
\cite{RS}, Ayad \cite{MA}, Swart \cite{CS}, and Silverman \cite[Theorem
14]{JS3}. In \cite{GB}, the authors used a result of Ayad
\cite[Th\'{e}or\`{e}me C]{MA} to show that the sequences of values of the
division polynomials $(\Phi_{n}(P)\text{ mod }p^{l})_{n\geq0}$ and
$(\overline{\Omega}_{n}(P)\text{ mod }p^{l})_{n\geq0}$ are purely periodic. In
this section, we generalize these results to the sequences of values of the
translated division polynomials. However, we first need a result of Swart
\cite[Theorem 9.3.2]{CS}.

\begin{theorem}
[\cite{CS}]Let $(W(n))_{n\geq0}$ be a Somos $4$ sequence with $\tau_{1}$,
$W(0)\neq0$, let $E/%
\mathbb{Q}
$ be the associated elliptic curve, and let $\mathbf{P}=(P_{1},P_{2})\in$
$E(\mathbb{Q})^{2}$ be the associated pair of rational points consisting of
points of infinite order. Let $p$ be an odd prime such that $P_{1}$ and
$P_{2}~$modulo $p$ are nonsingular, let the order of $P_{1}$ modulo $p$ be
$r>3$, and let $p^{v}$ be the highest power of $p$ such that $[r]P_{1}%
\equiv\mathcal{O}$ $(\operatorname{mod}$ $p^{v})$. Assume further that $P_{2}%
$, $P_{1}\pm P_{2}\not \equiv $ $\mathcal{O}$ $(\operatorname{mod}p)$ and that
ord$_{p}\left(  W(0)\right)  $ $\geq0$, ord$_{p}\left(  \frac{W(1)}%
{W(0)}\right)  $ $=0$. For every positive integer $l$, write $r_{l}=p^{l-v}r$
for $l\geq v$ and $r_{l}=r $ for $l<v$. Then, there exist integers $u_{l}$ and
$v_{l}$, relatively prime to $p$, such that for all nonnegative integers $k$,
$n$ and for all $l\geq1$,
\begin{equation}
W(kr_{l}+n)\equiv u_{l}^{kn}(u_{l}^{r_{l}})^{k^{2}/2}v_{l}^{k}{}%
~W(n)~(\operatorname{mod}\text{ }p^{l})\label{y}%
\end{equation}
where $(u_{l}^{r_{l}})^{k^{2}/2\text{ }}$denotes $d_{l}^{k^{2}}$, and $d_{l}$
is a chosen unit modulo $p^{l}$ satisfying $d_{l}^{2}$ $\equiv u_{l}^{r_{l}%
}~(\operatorname{mod}$ $p^{l})$.
\end{theorem}

From this, Swart immediately deduces that $(W(n)~\operatorname{mod}%
p^{l})_{n\geq0}$ is periodic with period $r_{l}t_{l}$ for some integer $t_{l}$
dividing $p^{l-1}(p-1)$ for all positive integers $l$, see \cite[Theorem
9.4.1]{CS}. We apply the above theorem to the sequence $(\Psi_{n}%
(\mathbf{P}))_{n\geq0} $ to obtain the following symmetry congruence.

\begin{theorem}
Let $E/%
\mathbb{Q}
$ be an elliptic curve and let $\mathbf{P}=(P_{1},P_{2})$ $\in$ $E(%
\mathbb{Q}
)^{2}$ be the associated pair of rational points of infinite order. Let $p$ be
an odd prime at which the chosen integral Weierstrass model of $E$ has good
reduction, and let $r>3$ be the order of $P_{1}$ modulo $p$. Assume that
$P_{2}$, $P_{1}\pm P_{2}\not \equiv $ $\mathcal{O}$ $(\operatorname{mod}%
\text{~}p)$. Let $p^{v}$ be the highest power of $p$ such that $[r]P_{1}%
\equiv\mathcal{O}$ $(\operatorname{mod}$ $p^{v})$, and for every positive
integer $l$, write $r_{l}=p^{l-v}r$ for $l\geq v$, $r_{l}=r$ for $l<v$. Let
$(\Psi_{n}(\mathbf{P}))_{n\geq0}$ be the sequence of values of the translated
division polynomials of $E$ at $\mathbf{P}$. Then, there exist integers
$a_{l}$ and $b_{l}$, relatively prime to $p$, such that
\begin{equation}
\Psi_{kr_{l}+n}(\mathbf{P})\equiv a_{l}^{kn}(a_{l}^{r_{l}})^{k^{2}/2}b_{l}%
^{k}{}~\Psi_{n}(\mathbf{P})~(\operatorname{mod}p^{l})\label{11}%
\end{equation}
for all nonnegative integers $k$, $n$ and for all $l\geq1$, where
$(a_{l}^{r_{l}})^{k^{2}/2\text{ }}$denotes $d_{l}^{k^{2}}$, and $d_{l}$ is a
chosen unit modulo $p^{l}$ satisfying $d_{l}^{2}$ $\equiv a_{l}^{r_{l}%
}~(\operatorname{mod}$ $p^{l})$.
\end{theorem}

\begin{proof}
For every positive integer $l$, $r_{l}$ is the order of $P_{1}$ modulo $p^{l}
$. By (\ref{s}), the sequence $(\Psi_{n}(\mathbf{P}))_{n\geq0}$ is a Somos 4
sequence with parameters $\tau_{1}=\beta_{1}^{2}$ and $\tau_{2}=-\beta_{2}$
with $\Psi_{0}(\mathbf{P})=\Psi_{1}(\mathbf{P})=1$.

Applying Theorem 9.3.2 of \cite{CS} to this sequence gives units $\gamma_{l} $
and $-\beta_{l}$ modulo $p^{l}$ such that%
\[
\Psi_{kr_{l}+n}(\mathbf{P})\equiv\gamma_{l}^{kn}(\gamma_{l}^{r_{l}%
})^{k(k-1)/2}(-\beta_{l}{})^{k}~\Psi_{n}(\mathbf{P})~(\text{mod~}%
p^{l})\text{.}%
\]
Set $a_{l}=\gamma_{l}$. Taking $k=1$ and then $n=0$ and $n=1$ gives%
\[
-\beta_{l}{}\equiv\Psi_{r_{l}}(\mathbf{P})\text{, }a_{l}{}\equiv\frac
{\Psi_{r_{l}+1}(\mathbf{P})}{\Psi_{r_{l}}(\mathbf{P})}~(\text{mod}%
~p^{l})\text{.}%
\]
Let $a\in\mathbb{F}_{p}^{\ast}$ be the constant given in Theorem 3.1. By its
proof, there exists $u$ $\in\mathbb{F}_{p}^{\ast}$ such that $u^{2}=a^{r}$.
Since $r_{l}=sr$ for a positive integer $s$, we have%
\[
a_{l}\equiv\frac{\Psi_{sr+1}(\mathbf{P})}{\Psi_{sr}(\mathbf{P})}\equiv
a^{s}~(\text{mod}~p)\text{.}%
\]
Thus%
\[
a_{l}^{r_{l}}\equiv(a_{l})^{sr}\equiv(a^{s})^{sr}\equiv(a^{r})^{s^{2}}%
\equiv(u^{2})^{s^{2}}\equiv(u^{s^{2}})^{2}(\text{mod}~p).
\]
Since $p$ is odd and $u\neq0$ there exists a unit $d_{l}$ modulo$~p^{l}$ such
that $a_{l}^{r_{l}}\equiv d_{l}^{2}$ $($mod$~p^{l})$, by Hensel's Lemma. Now
choose $b_{l}$ so that $b_{l}\equiv-\beta_{l}d_{l}^{-1}$ $($mod$~p^{l})$. Then
$b_{l}$ is a unit and
\begin{align*}
a_{l}^{kn}(a_{l}^{r_{l}})^{k(k-1)/2}(b_{l}{}d_{l}{})^{k}  & \equiv a_{l}%
^{kn}(d_{l}^{2})^{k(k-1)/2}(b_{l}{}d_{l}{})^{k}\\
& \equiv a_{l}^{kn}(d_{l})^{k^{2}}b_{l}^{k}{}\\
& \equiv a_{l}^{kn}(a_{l}^{r_{l}})^{k^{2}/2}b_{l}^{k}{}~(\text{mod~}%
p^{l})\text{,}%
\end{align*}
since $a_{l}^{r_{l}}\equiv d_{l}^{2}$ $($mod$~p^{l})$. Thus%
\[
\Psi_{kr_{l}+n}(\mathbf{P})\equiv a_{l}^{kn}(a_{l}^{r_{l}})^{k^{2}/2}b_{l}%
^{k}{}~\Psi_{n}(\mathbf{P})~(\text{mod~}p^{l})
\]
as required.
\end{proof}

Theorem 8.5.1 of \cite{CS} and this theorem show that $(\Psi_{n}%
(\mathbf{P})~\text{mod }p^{l})_{n\geq0}$ is periodic with period $r_{l}t_{l}$
for some integer $t_{l}$ dividing $p^{l-1}(p-1)$ for all $l\geq1$. The
following theorem shows that the corresponding symmetry congruences hold for
$(\Phi_{n}(\mathbf{P}))_{n\geq0}$ and $(\overline{\Omega}{}{}{}_{n}%
(\mathbf{P}))_{n\geq0}$.

\begin{theorem}
With the assumptions and notation as in Theorem 5.2, let $(\Phi_{n}%
(\mathbf{P}))_{n\geq0}$, and $(\overline{\Omega}{}{}{}_{n}(\mathbf{P}%
))_{n\geq0}$ be the sequences of values of the translated division polynomials
of $E$ at $\mathbf{P}$ as in (\ref{2}). Then,
\begin{equation}
\Phi_{kr_{l}+n}(\mathbf{P})\equiv a_{l}^{2kn}(a_{l}^{r_{l}})^{k^{2}}b_{l}%
^{2k}{}~\Phi_{n}(\mathbf{P})~(\operatorname{mod}p^{l})\text{,}\label{12}%
\end{equation}
and
\begin{equation}
\overline{\Omega}_{kr_{l}+n}(\mathbf{P})\equiv a_{l}^{3kn}(a_{l}^{r_{l}%
})^{3k^{2}/2}b_{l}^{3k}{}\overline{\Omega}\,_{n}(\mathbf{P}%
)~(\operatorname{mod}p^{l})\text{ }\label{13}%
\end{equation}
for all nonnegative integers $k$, $n$ and for all $l\geq1$, with the square
root fixed as in Theorem 5.2.
\end{theorem}

\begin{proof}
The theorem can be proved by using equations (\ref{3.1}), (\ref{3.2}), and
(\ref{11}) as in the proof of Theorem 4.1.
\end{proof}

As an immediate consequence of this result, we deduce the periodicity of the
sequences $(\Phi_{\mathbf{n}}(\mathbf{P})\text{ mod }p^{l})_{n\geq0}$ and
$(\Omega\,_{\mathbf{n}}(\mathbf{P})\text{ mod }p^{l})_{n\geq0}$.

\begin{proof}
[Proof of Theorem 1.2]We can show that the sequences $(\Phi_{\mathbf{n}%
}(\mathbf{P})\text{ mod }p^{l})_{n\geq0}$, and $(\overline{\Omega
}\,_{\mathbf{n}}(\mathbf{P})\text{ mod }p^{l})_{n\geq0}$ are purely periodic
with periods $r_{l}t_{l}$ and $r_{l}t_{l}^{\prime}$ respectively, for some
integers $t_{l}$ and $t_{l}^{\prime}$ dividing $p^{l-1}(p-1)$ for all $l\geq
1$, by using the congruences (\ref{12}), (\ref{13}) as in the proof of
Corollary 4.2.
\end{proof}

\begin{theorem}
With the assumptions and notation as in Theorem 5.2, let $\varepsilon_{l}$,
$\kappa_{l}$, $\lambda_{l}$ and $\varepsilon_{l}^{\prime}$, $\kappa
_{l}^{\prime}$, $\lambda_{l}^{\prime}$ be the orders of $a_{l}^{2}$,
$b_{l}^{2}$, $a_{l}^{r_{l}}$ and $a_{l}^{3}$, $b_{l}^{3}$, $(a_{l}^{r_{l}}%
{})^{3/2}$ in $%
\mathbb{Z}
_{p^{l}}^{\ast}$, respectively, with the square root fixed as in Theorem 5.2.
Then, the sequences $(\Phi_{n}(\mathbf{P})\text{ }\operatorname{mod}%
p^{l})_{n\geq0}$ and $(\overline{\Omega}\,{}_{n}(\mathbf{P})$
$\operatorname{mod}p^{l})_{n\geq0}$ are purely periodic with periods
$r_{l}t(\varepsilon_{l}$,$~\kappa_{l}$,$~\lambda_{l})$ and $r_{l}%
t(\varepsilon_{l}^{\prime}$,$~\kappa_{l}^{\prime}$,$~\lambda_{l}^{\prime})$,
respectively, where $t$ is given by (\ref{t1}).
\end{theorem}

\begin{proof}
The proof follows from (\ref{12}) and (\ref{13}) by the same calculations as
in Theorem 4.3, since $%
\mathbb{Z}
_{p^{l}}^{\ast}$ is cyclic for odd $p$.
\end{proof}

\section{$p$-adic convergence of $(\Psi_{n}(\mathbf{P}))$, $(\Phi
_{n}(\mathbf{P}))$ and $(\overline{\Omega}_{n}(\mathbf{P}))$ Sequences}

In \cite{JS3}, Silverman used the Mazur-Tate $p$-adic $\sigma$-function to
prove the existence and algebraicity of the $p$-adic limit of certain
subsequences of the sequence of values of division polynomials $(\Psi
_{n}(P))_{n\geq1}$. More precisely, under the hypotheses of \cite[Theorem
2]{JS3}, it was shown that if $E$ is an elliptic curve defined over $%
\mathbb{Q}
_{p}$ with good ordinary reduction and $P$ $\in E(%
\mathbb{Q}
_{p})$, then there is a power $q=p^{N}$ such that the sequence $(\Psi_{mq^{i}%
}(P))_{_{i\geq1}}$ converges in $\mathbb{Z}_{p}$ as $i\rightarrow\infty$ for
all $m\geq1$. Furthermore, in an addendum to \cite{JS3}, Silverman used the
results in \cite[Th\'{e}or\`{e}mes B and C]{MA} to prove that the sequence
$(\Psi_{mq^{i}}(P))_{_{i\geq1}}~$is $\mathbb{Z}_{p}$-Cauchy without the
ordinary reduction hypothesis. In \cite{GB}, the authors considered similar
problems for the sequences $(\Phi_{n}(P))_{n\geq1}$ and $(\overline{\Omega
}_{n}(P))_{n\geq1}$ of values of the division polynomials evaluated at a point
$P$ $\in$ $E(K)$ and used the periodicity properties of these sequences to
show that the sequences $(\Phi_{mq^{i}}(P))_{i\geq1}$, and $(\overline{\Omega
}_{mq^{i}}(P))_{i\geq1}$ are $\mathbb{Z}_{p}$-Cauchy without the ordinary
reduction hypothesis.

In this section, we prove that certain subsequences of the sequences
$(\Psi_{n}(\mathbf{P}))_{_{n\geq1}}$, $(\Phi_{n}(\mathbf{P}))_{_{n\geq1}}$,
and $(\overline{\Omega}{}_{n}(\mathbf{P}))_{_{n\geq1}}$ of values of the
translated division polynomials evaluated at a point $\mathbf{P}$ $\in$
$E(K)^{2}$ are $\mathbb{Z}_{p}$-Cauchy. We now use the symmetry congruences to
prove Theorem 1.3, following the congruence method in Silverman's addendum
\cite{JS3}.

\begin{proof}
[Proof of Theorem 1.3.]Let $p$ be an odd prime with $p\nmid r$. Since $p\nmid
r$, we may choose a positive integer $N$ such that $q=p^{N}$ $\equiv
1~(\text{mod }(p-1)r)$. First suppose that $[r]P_{1}\neq\mathcal{O}$, and let
$p^{v}$ be the highest power of $p$ such that $[r]P_{1}\equiv\mathcal{O}$
$(\text{mod}$ $p^{v})$. Then by (\ref{11}), there are integers $a_{l}$,
$b_{l}$, relatively prime to $p$, such that for all nonnegative integers $k$,
$n$, and all $l\geq v$%
\[
\Psi_{kp^{l-v}r+n}(\mathbf{P})\equiv a_{l}^{kn}(a_{l}^{r_{l}})^{k^{2}/2}%
b_{l}^{k}{}\Psi_{n}(\mathbf{P})~(\text{mod }p^{l})\text{.}%
\]
Then
\[
\Psi_{mq^{i}}(\mathbf{P})=\Psi_{mq^{j}(1+(p-1)rt)}(\mathbf{P})
\]
for all $i>j\geq2l/N$, all $l>v$, and some $t\in%
\mathbb{Z}
$, since $q^{i-j}\equiv1~(\text{mod }(p-1)r)$. Therefore we obtain%
\begin{align*}
\Psi_{mq^{i}}(\mathbf{P})  & =\Psi_{mt(p-1)p^{Nj-l+v}p^{l-v}r+mq^{j}%
}(\mathbf{P})\\
& =\Psi_{sp^{l-v}r+mq^{j}}(\mathbf{P})
\end{align*}
where $s=mt(p-1)p^{Nj-l+v}$. Thus by (\ref{11}), we have%
\[
\Psi_{mq^{i}}(\mathbf{P})\equiv a_{l}^{smq^{j}}(a_{l}^{r_{l}})^{s^{2}/2}%
b_{l}^{s}\Psi_{mq^{j}}(\mathbf{P})~(\text{mod }~p^{l})\text{,}%
\]
where $(a_{l}^{r_{l}})^{1/2}$ is the unit square root fixed in Theorem 5.2.
Thus%
\[
\Psi_{mq^{i}}(\mathbf{P})\equiv\Psi_{mq^{j}}(\mathbf{P})~(\operatorname{mod}%
~p^{l})\text{,}%
\]
since $s$ is a multiple of $~p^{l-1}(p-1)$. Therefore%
\[
\Vert\Psi_{mq^{i}}(\mathbf{P})-\Psi_{mq^{j}}(\mathbf{P})\Vert_{p}\leq
q^{-j/2+1}\text{ \ \ for all }i>j>2(v/N+1)
\]
where $j=\lceil2l/N\rceil$ and so the sequence $(\Psi_{mq^{i}}(\mathbf{P}))$
is $\mathbb{Z}_{p}$-Cauchy. Thus the sequence converges in $\mathbb{Z}_{p}$.
Similarly we can use the formulas (\ref{12}) and (\ref{13}) to show that the
sequences $(\Phi_{mq^{i}}(\mathbf{P}))$ and $(\overline{\Omega}_{mq^{i}%
}(\mathbf{P}))$ converge in $\mathbb{Z}_{p}$.

If $[r]P_{1}=\mathcal{O}$, the calculation in Section 3 gives the same
symmetry formulas modulo $p^{l}$ with unit coefficients. Hence each of the
three sequences has period $rp^{l-1}(p-1)$ modulo $p^{l}$. Since $q\equiv1$
$(\operatorname{mod}$ $r(p-1))$, this period divides $m(q^{i}-q^{j})$ whenever
$i>j$ and $Nj$ $\geq$ $l-1$. Thus the three subsequences are $\mathbb{Z}_{p}%
$-Cauchy also in this case, which completes the proof.
\end{proof}

\section{$p$-adic properties of Somos Sequences}

Somos sequences are generalizations of elliptic divisibility sequences and
there has been considerable interest in the study of these sequences. For
background on these recurrences and their integrality properties, see
\cite{FomZ, Gale1, Gale2, AH1, AH2, AH3, Mal, GB1}; for periodicity results,
see \cite{Robinson}, \cite{CS}.

Hone, in a series of papers \cite{AH1, AH3}, presented a simple algorithmic
approach to determining elliptic curves associated to Somos $4$ and Somos $5$
sequences. He proved that the terms of a Somos $4$ sequence correspond to a
sequence of points $[n]P_{1}+P_{2}$ on an associated elliptic curve $E$ and
can be expressed, up to factors, in terms of the sequence $(\Psi
_{n}(\mathbf{P}))_{n\geq0}$ of values of the translated division polynomials
where $\mathbf{P}=(P_{1},P_{2})$, (see also \cite{CS} for another approach).
He also noted that it would be interesting to see how some of his results on
Somos $4$ and Somos $5$ sequences could be applied to the $p$-adic setting.
Indeed, studies on the sigma function solution of the initial value problems
for Somos sequences play a key role in the $p$-adic convergence of Somos
sequences. We apply Theorem 1.3 to the Somos $4$ and Somos $5$ sequences
defined in (\ref{1.1}), using their elliptic representations and the
Teichm\"{u}ller character.

Throughout this section, the translated division polynomials are evaluated at
pairs of points satisfying the evaluation conditions stated in Section 2.

We first consider Somos $4$ sequences:\ Hone \cite[Theorem 1.1]{AH1} proved
that the terms of a Somos $4$ sequence can be given in terms of the
Weierstrass sigma function for an associated elliptic curve: let $W(n)$ be a
Somos $4$ sequence with $\tau_{1}$, $W(0)\neq0$. Then the sequence $W(n)$
corresponds to a sequence of points $[n]P_{1}+P_{2}$ on an associated elliptic
curve $E$. Moreover, there exist a lattice $L\subset%
\mathbb{C}
$ and non-zero complex numbers $z_{1}$ and $z_{2}$ such that
\begin{align}
W(n)  & =\frac{W(0)}{\sigma(z_{2})}\left(  \frac{\sigma(z_{1})\sigma
(z_{2})W(1)}{\sigma(z_{1}+z_{2})W(0)}\right)  ^{n}\frac{\sigma(nz_{1}+z_{2}%
)}{\sigma(z_{1})^{n^{2}}}\nonumber\\
& =\frac{W(1)^{n}}{W(0)^{n-1}}\Omega_{n}(z_{1},z_{2}\mathbf{;~}L)\nonumber\\
& =\alpha^{n}\beta\Omega_{n}(z_{1},z_{2}\mathbf{;~}L)\label{z2}%
\end{align}
for all $n\geq1$, where $\alpha=\frac{W(1)}{W(0)}$, $\beta=W(0)$ and
$\Omega_{n}(\mathbf{z;~}L)$ is the analytic translated $n$-division function
defined as in (\ref{7}).

In the following theorem, we reformulate Hone's result \cite[Theorem 1.1]{AH1}
in terms of the sequence of values of the translated division polynomials.

\begin{theorem}
Let $W(n)$ be a Somos $4$ sequence with $\tau_{1}$, $W(0)\neq0$ and$~$let $E/%
\mathbb{Q}
$ and $\mathbf{P}=(P_{1}$, $P_{2})$ $\in$ $E(%
\mathbb{Q}
)^{2}$ be the associated elliptic curve and pair of rational points. Let
$(\Psi_{n}(\mathbf{P}))_{n\geq0}$ be the sequence of values of translated
division \textit{polynomials} associated to $E$. Then%
\begin{equation}
W(n)=\alpha^{n}\beta\Psi_{n}(\mathbf{P})\label{14}%
\end{equation}
for all $n\geq1$, where $\alpha=\frac{W(1)}{W(0)}$, $\beta=W(0)$.
\end{theorem}

\begin{proof}
The proof follows from \cite[Theorem 1.1]{AH1}, and (\ref{10}).
\end{proof}

Under the hypotheses of Theorem 1.4, we have $\alpha$, $\beta$ $\in
\mathbb{Z}_{p}$, and $\Psi_{n}(\mathbf{P})\in\mathbb{Z}_{p}$ by Section 2.
Hence (\ref{14}) gives%
\[
\text{ord}_{p}(W(n))=n~\text{ord}_{p}(\alpha)+\text{ord}_{p}(\beta
)+\text{ord}_{p}(\Psi_{n}(\mathbf{P}))\geq0
\]
for all $n\geq1$.

Now we can give the proof of Theorem 1.4. The proof uses the Teichm\"{u}ller
character for computing the limit $\lim\limits_{i\rightarrow\infty}W(mq^{i})$.
The classical\textit{\ Teichm\"{u}ller character} is the unique homomorphism%
\[
\chi:\mathbb{Z}_{p}^{\ast}\rightarrow\mathbf{\mu}_{p-1}\text{, \ \ satisfying
\ }\chi(a)\equiv a(\text{mod }p)\text{,}%
\]
and can be computed as the limit $\chi(a)=\lim\limits_{i\rightarrow\infty} $
$a^{p^{i}}$. This construction can be generalized to group schemes $G$ over
$\mathbb{Z}_{p}$, or other complete local rings, see \cite[Proposition
10]{JS3} for more details. See also \cite[Sections 3, 4, and 11]{HC} for
background, definitions, and all properties of the Teichm\"{u}ller character.

\begin{proof}
[Proof of Theorem 1.4]As noted above, $\Psi_{n}(\mathbf{P})\in\mathbb{Z}_{p}$
and $\alpha$, $\beta$ $\in\mathbb{Z}_{p}$. On the other hand, ord$_{p}%
(\alpha)\geq0$ and ord$_{p}(\beta)\geq0$ by the hypotheses. Thus, if
ord$_{p}(\alpha)>0$, then by (\ref{14}) we obtain
\[
\text{ord}_{p}(W(n))\geq n~\text{ord}_{p}(\alpha)+\text{ord}_{p}%
(\beta)\rightarrow\infty\ \ \ \text{as }n\rightarrow\infty.
\]
It follows that $\lim\limits_{n\rightarrow\infty}W(n)=0$ in $\mathbb{Z}_{p} $,
and hence the required limit is zero.

Now consider the case $\alpha\in\mathbb{Z}_{p}^{\ast}$. Choose $q=p^{N}$ as in
Theorem 1.3. Then, for all $m\geq1$,
\begin{equation}
\lim\limits_{i\rightarrow\infty}\alpha^{mq^{i}}\beta=\chi(\alpha)^{m}%
\beta\text{ \ in }\mathbb{Z}_{p}\label{15}%
\end{equation}
where $\chi(\alpha)\in\mathbf{\mu}_{p-1}$ is the value of the Teichm\"{u}ller
character. Thus, by (\ref{14}) and Theorem 1.3, we obtain the existence of the
limit
\[
\lim\limits_{i\rightarrow\infty}W(mq^{i})=\chi(\alpha)^{m}\beta\lim
\limits_{i\rightarrow\infty}\Psi_{mq^{i}}(\mathbf{P})\text{ in }\mathbb{Z}%
_{p}\text{,}%
\]
which completes the proof.
\end{proof}

We now consider Somos $5$ Sequences: Hone \cite[Theorem 2.9]{AH3} also proved
that the terms of a Somos $5$ sequence can be expressed in terms of the
Weierstrass sigma function for an associated elliptic curve: let $W(n)$ be a
Somos $5$ sequence with $W(0)W(1)\neq0$. Then there exist a lattice $L\subset%
\mathbb{C}
$ and non-zero complex numbers $z_{1}$, $z_{2}$ such that
\begin{align*}
W(2n)  & =\frac{W(0)}{\sigma(z_{2})}\left(  \frac{\sigma(2z_{1})\sigma
(z_{2})W(2)}{\sigma(2z_{1}+z_{2})W(0)}\right)  ^{n}\frac{\sigma(2nz_{1}%
+z_{2})}{\sigma(2z_{1})^{n^{2}}}\\
& =\frac{W(2)^{n}}{W(0)^{n-1}}\Omega_{n}(2z_{1},z_{2}\mathbf{;~}L)\\
& =\alpha^{n}\beta\Omega_{n}(2z_{1},z_{2}\mathbf{;~}L)\text{,}%
\end{align*}
and
\begin{align*}
W(2n+1)  & =\frac{W(1)}{\sigma(z_{1}+z_{2})}\left(  \frac{\sigma(2z_{1}%
)\sigma(z_{1}+z_{2})W(3)}{\sigma(3z_{1}+z_{2})W(1)}\right)  ^{n}\frac
{\sigma((2n+1)z_{1}+z_{2})}{\sigma(2z_{1})^{n^{2}}}\\
& =\frac{W(3)^{n}}{W(1)^{n-1}}\Omega_{n}(2z_{1},z_{1}+z_{2}\mathbf{;~}L)\\
& =\gamma^{n}\delta\Omega_{n}(2z_{1},z_{1}+z_{2}\mathbf{;~}L)
\end{align*}
for all $n\geq1$, where $\alpha=\frac{W(2)}{W(0)}$, $\beta=W(0)$,
$\gamma=\frac{W(3)}{W(1)}$, $\delta=W(1)$, and $\Omega_{n}(\mathbf{z;~}L)$ is
the analytic translated $n$-division function defined as in (\ref{7}). Hence,
by (\ref{z2}) and (\ref{10}) the even and odd indexed subsequences are each
expressed as a factor times a translated division polynomial sequence. Each
subsequence also satisfies a Somos $4$ recurrence.

As for Somos $4$, this gives the following representation in terms of
sequences of values of translated division polynomials.

\begin{theorem}
Let $W(n)$ be a Somos $5$ sequence with $W(0)W(1)\neq0$,$~$let $E/%
\mathbb{Q}
$ be the associated elliptic curve, and let $\mathbf{P}=(P_{1},P_{2})$ and
$\mathbf{Q}=(Q_{1},Q_{2})\in$ $E(\mathbb{Q})^{2}$ be pairs of rational points
associated to the even and odd indexed subsequences of $(W(n))$, respectively.
Let $(\Psi_{n}(\mathbf{P}))_{n\geq0}$ and $(\Psi_{n}(\mathbf{Q}))_{n\geq0}$ be
the sequences of values of translated division \textit{polynomials} associated
to $E$. Then%
\begin{equation}
W(2n)=\alpha^{n}\beta\Psi_{n}(\mathbf{P})\text{ \ \ and \ }W(2n+1)=\gamma
^{n}\delta\Psi_{n}(\mathbf{Q})\label{z4}%
\end{equation}
for all $n\geq1$, where $\alpha=\frac{W(2)}{W(0)}$, $\beta=W(0)$,
$\gamma=\frac{W(3)}{W(1)}$, $\delta=W(1)$.
\end{theorem}

\begin{proof}
The proof follows from \cite[Theorem 2.9]{AH3}, and (\ref{10}).
\end{proof}

Under the hypotheses of Theorem 1.5, the four constants $\alpha$, $\beta$,
$\gamma$, $\delta$ belong to $\mathbb{Z}_{p}$. The integrality observation in
Section 2 and (\ref{z4}) therefore imply that both subsequences $(W(2n))$ and
$(W(2n+1))$ consist of $p$-adic integers.

Finally, we can use Theorem 7.2 and Theorem 1.3 to prove the $p$-adic
convergence of Somos $5$ sequences.

\begin{proof}
[Proof of Theorem 1.5]The proof is similar to that of Theorem 1.4, using
equation (\ref{z4}). Let $p^{N_{1}}$ and $p^{N_{2}}$ be the powers given by
Theorem 1.3 for $\mathbf{P} $ and $\mathbf{Q}$, respectively. Now take $N=$
lcm($N_{1},N_{2}$), $q=p^{N}$. Since $N_{1}$ and $N_{2}$ both divide $N$, the
sequences $\Psi_{mq^{i}}(\mathbf{P})$ and $\Psi_{mq^{i}}(\mathbf{Q})$ are
subsequences of the corresponding convergent sequences. Thus the same $q$ can
be used for both limits, and the result follows from (\ref{z4}) by the
argument used in the proof of Theorem 1.4.
\end{proof}

Finally, we investigate the periodicity properties of Somos $5$ sequences
modulo prime powers. We note that the symmetry formulas for Somos $5$
sequences associated to elliptic curves can be obtained as follows. We next
impose the unit conditions on $\frac{W(2)}{W(0)}$ and $\frac{W(3)}{W(1)}$. The
geometric factors in (\ref{z4}) are then periodic modulo every power of $p$.
Combining (\ref{z4}) with the symmetry congruences for the sequences of values
of translated division polynomials gives the following result.

\begin{theorem}
Let $W(n)$ be a Somos $5$ sequence with $W(0)W(1)\neq0$ and let $E/%
\mathbb{Q}
$ be the associated elliptic curve. Let $\mathbf{P}=(P_{1},P_{2})$ and
$\mathbf{Q}=(Q_{1},Q_{2})\in$ $E(\mathbb{Q})^{2}$ be pairs of rational points
associated to the even and odd indexed subsequences of $(W(n))$, respectively.
Let $p$ be an odd prime at which the chosen integral Weierstrass model of $E$
has good reduction. Let the orders of $P_{1}$ and $Q_{1}$ modulo $p$ be
$r_{1}$ and $r_{1}^{\prime}>3$, respectively. Assume that $[r_{1}]P_{1}%
\neq\mathcal{O~}$and $[r_{1}^{\prime}]Q_{1}\neq\mathcal{O}$. Let $p^{v}$ and
$p^{w}$ be the highest powers of $p$ such that $[r_{1}]P_{1}\equiv\mathcal{O}$
$(\operatorname{mod}$ $p^{v})$ and $[r_{1}^{\prime}]Q_{1}\equiv\mathcal{O}$
$(\operatorname{mod}$ $p^{w})$, respectively, and write $r_{l}=p^{l-v}r_{1}$
for $l\geq v$ and $r_{l}=r_{1}$ for $l<v$, $r_{l}^{\prime}=p^{l-w}%
r_{1}^{\prime}$ for $l\geq w$ and $r_{l}^{\prime}=r_{1}^{\prime}$ for $l<w$.
Assume further that $P_{1}$, $P_{2}$, $P_{2}\pm P_{1}$, $Q_{1}$, $Q_{2}$,
$Q_{2}\pm Q_{1}\not \equiv $ $\mathcal{O}$ $(\operatorname{mod}p)$, and that
ord$_{p}\left(  W(0)\right)  $, ord$_{p}\left(  W(1)\right)  $ $\geq0$,
ord$_{p}\left(  \frac{W(2)}{W(0)}\right)  $, ord$_{p}\left(  \frac{W(3)}%
{W(1)}\right)  $ $=0$. Then there exist integers $a_{l}$, $b_{l}$, $c_{l}$,
$d_{l}$ relatively prime to $p$, such that
\begin{align}
W(2kr_{l}+2n)  & \equiv a_{l}^{kn}(a_{l}^{r_{l}})^{k^{2}/2}{}b_{l}^{k}%
{}W(2n)\text{ }(\operatorname{mod}p^{l})\text{,}\label{u}\\
W(2kr_{l}^{\prime}+2n+1)  & \equiv c_{l}^{kn}(c_{l}^{r_{l}^{\prime}}%
)^{k^{2}/2}{}d_{l}^{k}{}{}W(2n+1)\text{ }(\operatorname{mod}p^{l})\label{v}%
\end{align}
for all nonnegative integers $k$, $n$ and all $l\geq1$.
\end{theorem}

\begin{proof}
The proof uses Theorem 5.2 and equation (\ref{z4}). Since $p$ is an odd prime
of good reduction, a nonzero torsion point in $E(%
\mathbb{Q}
)$ cannot reduce to $\mathcal{O}$ modulo $p$. Thus, as $[r_{1}]P_{1}$ and
$[r_{1}^{\prime}]Q_{1}$ are nonzero and reduce to $\mathcal{O}$ modulo $p$,
both $P_{1}$ and $Q_{1}$ have infinite order. The symmetry congruences of
Theorem 5.2 remain valid when the second point is torsion. Indeed, for a fixed
$l$, such a point $P_{2}$ can be replaced by $P_{2}+[r_{l}]P_{1}$, which has
infinite order and is congruent to $P_{2}$ modulo $p^{l}$. By the algebraic
construction in Section 2 and the reduction hypotheses, this replacement
preserves the values of the translated division polynomials modulo $p^{l}$.
Thus the congruences of Theorem 5.2 hold for $\mathbf{P}$. The same argument
applies to $\mathbf{Q}$. The result now follows from (\ref{z4}), since
$\alpha$ and $\gamma$ are $p$-adic units and $\beta$, $\delta\in\mathbb{Z}%
_{p}$.
\end{proof}

As an immediate consequence of this result, we deduce the periodicity of the
subsequences $(W(2n)\text{ mod }p^{l})$ and $(W(2n+1)\text{ mod }p^{l})$ which
proves Theorem 1.6.

\begin{corollary}
With the assumptions and notation as in Theorem 7.3, even and odd indexed
subsequences of a Somos $5$ sequence are purely periodic with periods
$r_{l}t_{l}$ and $r_{l}^{\prime}t_{l}^{\prime}$ respectively, for some
integers $t_{l}$ and $t_{l}^{\prime}$ dividing $p^{l-1}(p-1)$ for all positive
integers $l$.
\end{corollary}

\begin{proof}
This can be proved using equations (\ref{u}) and (\ref{v}) as in the proof of
Corollary 4.2.
\end{proof}

The following theorem gives explicit formulas for computing periods of the
even and odd indexed subsequences of a Somos $5$ sequence.

\begin{theorem}
With the assumptions and notation as in Theorem 7.3, let $\varepsilon_{l}$,
$\kappa_{l}$, $\lambda_{l}$ and $\varepsilon_{l}^{\prime}$, $\kappa
_{l}^{\prime}$, and $\lambda_{l}^{\prime}$ denote the orders of $a_{l}$,
$b_{l}$, $(a_{l}^{r_{l}})^{1/2}$ and $c_{l}$, $d_{l}$, $(c_{l}^{r_{l}^{\prime
}})^{1/2}$ in $%
\mathbb{Z}
_{p^{l}}^{\ast}$, respectively, with the square roots used in the symmetry
formulas of Theorem 7.3. Then the sequences $(W(2n)\text{ }\operatorname{mod}%
p^{l})$ and $(W(2n+1)$ $\operatorname{mod}p^{l})$ are purely periodic with
periods $r_{l}t(\varepsilon_{l}$,$~\kappa_{l}$,$~\lambda_{l})$ and
$r_{l}^{\prime}t(\varepsilon_{l}^{\prime},~\kappa_{l}^{\prime},~\lambda
_{l}^{\prime}) $, respectively, where $t$ is given by (\ref{t1}).
\end{theorem}

\begin{proof}
This can be proved using equations (\ref{u}) and (\ref{v}) as in the proof of
Theorem 4.3.
\end{proof}

\textbf{Acknowledgement.} This work was supported by the Scientific and
Technological Research Council of T\"{u}rkiye (T\"{U}B\.{I}TAK) under Project
no: 118F322.

\end{document}